\documentclass[pdflatex,sn-mathphys-num]{sn-jnl}

\usepackage{mathrsfs}
\usepackage{graphicx}%
\usepackage{multirow}%
\usepackage{amsmath,amssymb,amsfonts}%
\usepackage{amsthm}%
\usepackage{mathrsfs}%
\usepackage[title]{appendix}%
\usepackage{xcolor}%
\usepackage{textcomp}%
\usepackage{manyfoot}%
\usepackage{booktabs}%
\usepackage{algorithm,algorithmic}
\usepackage{graphicx}
\usepackage{epstopdf}
\usepackage{thmtools}%
\usepackage{siunitx}

\usepackage{booktabs}
\usepackage{multirow}      
\usepackage{makecell}     
\usepackage{amsopn}
\usepackage{hyperref}
\usepackage{cleveref}
\usepackage{caption}
\usepackage{subcaption}

\usepackage{listings}%

\theoremstyle{thmstyleone}%
\newtheorem{theorem}{Theorem}
\newtheorem{lemma}[theorem]{Lemma}%
\newtheorem{proposition}[theorem]{Proposition}%

\theoremstyle{thmstyletwo}%
\newtheorem{example}{Example}%

\theoremstyle{thmstylethree}%
\newtheorem{definition}{Definition}%

\newtheoremstyle{thmstylefour}{}{}{\itshape}{}{\bfseries}{.}{ }{}
\theoremstyle{thmstylefour}
\newtheorem{assumption}{Assumption}

\DeclareMathOperator{\argmin}{argmin}

\begin{document}

\title[Article Title]{An Infeasible Method with Feasibility Safeguard for Nonsmooth Composite Optimization Over Manifolds}


\author[1]{\fnm{Xiyuan} \sur{Xie}}\email{xiexy65@mail2.sysu.edu.cn}

\author*[2]{\fnm{Qia} \sur{Li}}\email{liqia@mail.sysu.edu.cn}


\affil[1]{\orgdiv{School of Mathematics}, \orgname{Sun Yat-sen University}, \orgaddress{\city{Guangzhou}, \postcode{510275}, \country{China}}}

\affil*[2]{\orgdiv{School of Computer Science and Engineering, Guangdong Province Key Laboratory of Computational Science}, \orgname{Sun Yat-sen University}, \orgaddress{\city{Guangzhou}, \postcode{510275}, \country{China}}}



\abstract{We study nonsmooth composite optimization over embedded
submanifolds defined by nonlinear equality constraints.
We propose a feasibility-safeguarded inexact proximal
linearized method whose iterations require neither
projection nor retraction onto the manifold.
A computable smooth-dual residual certifies the
approximate tangential step and controls the errors in
both tangency and composite linearization consistency.
The step is followed by a normal correction evaluated
at the displaced trial point. To coordinate the resulting
optimality and feasibility errors, a nonmonotone
backtracking procedure reduces the tangential stepsize
faster than the normal stepsize and accepts only corrected
points within a prescribed constraint-violation region.
Under the stated regularity and level-boundedness
assumptions, we establish a uniform backtracking bound,
positive lower bounds on the accepted stepsizes,
boundedness of the primal and multiplier sequences,
stationarity of all cluster points, and an
$O(\varepsilon^{-2})$ outer iteration complexity bound.
The analysis does not assume compactness of the manifold
or a separate global distance-to-manifold error bound.
For a specified parameter regime guaranteeing full
tangential steps, a KL assumption on an auxiliary
function yields full-sequence convergence.
Numerical experiments on sparse principal component
analysis and sparse spectral clustering illustrate the
performance of the method.}

\keywords{Nonsmooth manifold optimization, Infeasible method, Feasibility safeguard, Kurdyka--\L{}ojasiewicz property} 



\maketitle

\section{Introduction}
Let $\mathcal{M} \subseteq \mathbb{R}^n$ be a submanifold embedded in $\mathbb{R}^n$ with a global defining function \mbox{$h: \mathbb{R}^n \to \mathbb{R}^m$}, i.e., $\mathcal{M} := \{ x \in \mathbb{R}^n : h(x) = 0 \}$. In this paper, we consider the following nonsmooth composite optimization over $\mathcal{M}$:
 \begin{equation}\label{prob:primal}
\min_{x\in \mathcal{M}}~~ F(x):=f(x) + g(\mathcal{A}(x))
\end{equation}
where $f: \mathbb{R}^n \to \mathbb{R}$ and $\mathcal{A}: \mathbb{R}^n \to \mathbb{R}^p$ are continuously differentiable, and $g: \mathbb{R}^p \to \mathbb{R}$ is closed and convex.
For $a \geq 0$, we denote by $\mathcal{M}_a$ the constraint-violation sublevel set of $\mathcal{M}$, that is,
\begin{equation*}
\mathcal{M}_a := \{ x \in \mathbb{R}^n : \|h(x)\| \leq a \}.
\end{equation*}
 We adopt the following blanket assumption on problem \eqref{prob:primal} throughout this paper.

\begin{assumption}\label{assp:N}
\hfill
\begin{enumerate}
    \item[(i)] $\nabla h(x):=[\nabla h_1(x),\dots,\nabla h_m(x)]^{\top}\in \mathbb{R}^{m\times n}$ is $L_h$-Lipschitz continuous over $\mathbb{R}^n$. Moreover, there exist $\theta > 0$ and $0 < C_1 \leq C_2$ such that for all $x \in \mathcal{M}_{\theta}$, the eigenvalues of $H(x) := \nabla h(x) \nabla h(x)^{\top}$ lie in $[C_1, C_2]$.
    \item[(ii)] $f$ is $\ell_f$-Lipschitz continuous and  $\nabla f$ is $L_f$-Lipschitz continuous over $\mathcal{M}_{\theta}$.
    \item[(iii)] $g$ is $\ell_g$-Lipschitz continuous over $\mathbb R^p$.
    \item[(iv)] $\mathcal{ A}$ is $\ell_{\mathcal{A}}$ -Lipschitz continuous and $\nabla \mathcal{A}$ is $L_{\mathcal{A}}$-Lipschitz continuous over $\mathcal{M}_{\theta}$.
    \item[(v)] The proximal operator $\mathrm{prox}_{t g}$ of $g$ is cheaply computable for every  $t>0$.
    \item[(vi)] For every $b\in\mathbb{R}$,  $\{x\mid F(x)\leq b\}\cap \mathcal{M}_{\theta}$ is bounded.
\end{enumerate}
\end{assumption}

We clarify the role of the assumptions and the region on which they
are imposed. The set $\mathcal M_\theta$ is defined directly by the
constraint residual, rather than by the distance to $\mathcal M$.
Membership in this set can therefore be checked by evaluating $h$.
Neither compactness of $\mathcal M$ nor boundedness of
$\mathcal M_\theta$ is assumed. Instead, Assumption~1(vi) requires
the objective sublevel sets within $\mathcal M_\theta$ to be bounded.
The safeguard keeps the accepted iterates in
$\mathcal M_\kappa\subset\mathcal M_\theta$, while the merit-function
acceptance test provides a uniform upper bound on their objective
values. These two properties imply boundedness of the primal
sequence; the multiplier bounds are obtained separately in Lemma~1.

The analysis retains uniform nondegeneracy of the constraint
Jacobian on $\mathcal M_\theta$. This requirement controls both the
normal correction and the multipliers and is stronger than pointwise
full rank on $\mathcal M$. The Lipschitz assumptions on $f$,
$\nabla f$, $\mathcal A$, and $\nabla\mathcal A$ are restricted to
$\mathcal M_\theta$, whereas the global Lipschitz assumptions on
$\nabla h$ and $g$ are retained. No separate global
distance-to-manifold error bound is imposed. Thus, the boundedness
result follows from explicit problem-level regularity and
level-boundedness conditions, rather than from a boundedness
assumption imposed directly on the generated sequence.
Appendix~A verifies Assumption~1(i) for the Stiefel and oblique
manifolds, as well as for two noncompact examples: affine subspaces
and elliptic cylinders.

Problem \eqref{prob:primal} has many important applications in machine learning and signal processing. Below, we present two representative examples and refer the interested readers to \citet{chenProximalGradientMethod2020} and \citet{BriefIntroductiontoManifoldOptimization} for more examples.  

\begin{example}
   \textbf{Sparse Principal Component Analysis (SPCA)}. The classical PCA (\citet{Hotelling1933Analysis}) is one of the most widely used dimensionality reduction techniques, which minimizes the total reconstruction error over all samples. Sparse PCA is further proposed to promote sparsity in the principal components (\citet{jolliffe2003modified}) as follows:
   \begin{equation}\label{SPCA}
       \min_{X\in \mathrm{St}(n, p)}-\frac{1}{2}\mathrm{Tr}(X^\top B^{\top} BX)+\mu \left\|X\right\|_1, 
   \end{equation}
where $B\in\mathbb{R}^{m\times n}$, $\mu>0$ is a weighting parameter and 
$\left\|X\right\|_1:=\sum_{i,j}\left|x_{ij}\right|$ is the $\ell_1$-norm of the matrix $X$.  Note that problem \eqref{SPCA} is an instance of  problem \eqref{prob:primal} with $\mathcal{M}$ being the Stiefel manifold $\mathrm{St}(n,p)=\{X\in\mathbb{R}^{n\times p}:X^{\top}X=I_p\}$,  $g(X)=\mu\|X\|_1$ and $f(X)=-\frac{1}{2}\mathrm{Tr}(X^\top B^{\top} BX)$.
\end{example}

\begin{example}\textbf{Sparse Spectral Clustering (SSC)}. The SSC problem aims to partition $n$ data samples into 
$p$ groups such that similar data points are clustered together. In spectral clustering, a symmetric affinity matrix $W = [W_{ij}]_{n\times n}$ is constructed, where $W_{ij}\geq 0$
 quantifies the pairwise similarity between samples $a_i$ and $a_j$. To promote sparsity and interpretability, the sparse spectral clustering  framework is introduced in \citet{2016ConvexSparseSpectralClustering},  as follows
\begin{equation}\label{SSC}
    \min_{X\in\mathrm{St}(n,p)} \mathrm{Tr}( L^{\top} XX^{\top})+\mu\|XX^{\top}\|_1,
\end{equation}
where $\mu > 0$ is the regularization parameter, $L = I_n -S^{-1/2}WS^{-1/2}$ is the normalized Laplacian matrix, 
$S^{1/2}$ is the diagonal matrix with diagonal elements $\sqrt{s_1}, \sqrt{s_2}, \dots, \sqrt{s_n}$
and $s_i = \sum_{j} W_{ij}$. 
Again, problem \eqref{SSC} is an instance of problem \eqref{prob:primal} with $\mathcal{M}$ being the Stiefel manifold $\mathrm{St}(n,p)$,  $f(X)=\mathrm{Tr}( L^{\top} XX^{\top})$ and $g(\mathcal{A}(X))=\mu\|XX^{\top}\|_1$.
\end{example}

\subsection{Related works}

Problem~\eqref{prob:primal} can equivalently be viewed as a nonsmooth composite
program with nonlinear equality constraints. Our use of manifold structure refers to the regular level-set geometry induced by $h$:
the tangential model is posed on the kernel of $\nabla h(x^k)$ at the $k$-th iteration,
and the normal correction uses
$\nabla h(y^k)^\top h(y^k)$ at tangential trial point $y^k$. At an infeasible iterate, this kernel
is the tangent space of the nearby level set
$\{x:h(x)=h(x^k)\}$. The method is therefore designed for the embedded, equality-defined setting and explicitly uses the defining function
$h$.

Prox-linear methods in Euclidean spaces provide an important
starting point. Their basic construction linearizes the smooth
inner mapping, retains the nonsmooth outer function, and adds a
proximal regularization term. \citet{LewisWright2016} developed a regularized linearization
framework for composite minimization and established convergence
and active-structure identification results. 
\citet{DrusvyatskiyLewis2018} related proximal step lengths to
near-stationarity and studied convergence under error-bound and
growth conditions. 
\citet{DrusvyatskiyPaquette2019} analyzed the efficiency of
prox-linear methods, including an $O(\varepsilon^{-2})$ guarantee
with exact subproblem solutions in their composite setting and
the effects of inexact subproblem solves and first-order inner
methods. Composite formulations can also encode constraints
through extended-valued outer functions. Thus, linearization,
proximal regularization, and inexact subproblem solution are not
features specific to the manifold setting.

The question addressed here is how to combine these ingredients
when the primal iterates may violate the nonlinear constraints
and the computed directions need not satisfy the linearized
constraints exactly, without using a projection or retraction
onto $\mathcal M$ in the iterative procedure. In this setting,
the errors from the tangential solve must be coordinated with
the normal correction and the acceptance tests. We next review
related manifold and equality-constrained methods, distinguishing
between methods that maintain feasibility and those that allow
infeasible iterates.

The first category consists of feasible methods, including Riemannian subgradient methods (\citet{grohs2016nonsmooth,hosseini2018line,hosseini2017riemannian}), operator splitting methods (\citet{chen2016augmented,kovnatsky2016madmm,lai2014splitting,li2025riemannian,Zhou2022A,zhu2017nonconvex}), Riemannian proximal gradient methods (\citet{chenProximalGradientMethod2020,huang2022riemannian,huang2023inexact,wang2022manifold,he2026inexact,zheng2025new,jiang2025inexact}), smoothing-type methods (\citet{beck2023dynamic,peng2023riemannian,zhu2024smoothing}), and the Riemannian minimax optimization algorithms (\citet{xie2025proximal,xu2026riemannian}). These methods keep the iterates on the manifold throughout the algorithmic process and therefore typically require a retraction step at each iteration. However, the cost of computing retractions may become significant in large-scale problems. For instance, on the Stiefel manifold, retraction can be expensive when the matrix dimension is large. Among the methods mentioned above, the augmented Lagrangian method (\citet{Zhou2022A}), (inexact) manifold proximal linear algorithm (\citet{wang2022manifold,zheng2025new}), smoothing-type methods (\citet{beck2023dynamic,peng2023riemannian,zhu2024smoothing}), and Riemannian minimax optimization algorithms (\citet{xie2025proximal,xu2026riemannian}) can be applied to the general model \eqref{prob:primal}, whereas many of the remaining methods are designed only for simplified settings in which $\mathcal{A}$ is the identity or a linear mapping.

The second category consists of infeasible methods, which allow infeasible iterates and thus avoid performing a retraction step at every iteration. From the viewpoint of constrained Euclidean optimization, problem~\eqref{prob:primal} can also be viewed as a nonsmooth composite optimization problem with nonlinear equality constraints. For this formulation, the augmented Lagrangian method and its variants are among the most commonly used approaches; see
\citet{bolte2018nonconvex,chen2017augmented,de2023constrained,hallak2023adaptive,lu2012augmented,xie2021complexity} for example. A major drawback of these methods is that their subproblems are
usually highly nonconvex due to the nonlinear constraints and can therefore be difficult to solve in practice. In addition, their convergence analyses often require boundedness of the multiplier sequence, which is typically imposed as an external assumption and may be difficult to verify. Very recently, inspired
by SQP methods for smooth constrained optimization
(\citet{chen2020penalty,gould2010nonlinear,liu2011sequential,ulbrich2003non}) and proximal gradient methods for unconstrained optimization, \citet{dai2025proximal} proposed a proximal gradient method (PG-SQP) for solving the special case of problem~\eqref{prob:primal} with $\mathcal A=I$. Compared with the augmented Lagrangian methods reviewed above, PG-SQP has computationally simpler subproblems and hence may enjoy lower per-iteration cost, but its analysis assumes the existence of an open convex set containing the
iterates and trial steps, on which $f$ is bounded below, $\nabla f$ and the constraint Jacobian are Lipschitz and bounded, and the subdifferential of the convex term is bounded.

Recently, several infeasible methods have been developed for the special case of problem~\eqref{prob:primal} in which $\mathcal A$ is the identity and $\mathcal M$ is the Stiefel manifold. These methods do not impose an additional boundedness assumption on the generated sequence. In particular, when $g$ is the
$\ell_{2,1}$-norm, \citet{xiao2021exact} developed an
exact-penalty convex constrained proximal gradient algorithm. Later, motivated by SQP methods for smooth equality-constrained optimization, \citet{liu2024penalty} proposed a sequential linearized proximal gradient method (SLPG), which alternates between tangential steps and normal steps in order to improve optimality and feasibility, respectively. More recently, \citet{hu2024constraint} proposed a constraint-dissolving reformulation for general nonsmooth optimization over the Stiefel manifold. Based on this reformulation, they developed a stochastic proximal subgradient algorithm for the specialization of problem \eqref{prob:primal} with $\mathcal{A}=I$ over the Stiefel manifold. Their algorithm further allows the component corresponding to $f$ in problem \eqref{prob:primal} to be nonsmooth and merely locally Lipschitz continuous. A common feature of these algorithms is that they control infeasibility through an explicit upper bound on the stepsize, so that the iterates remain in a prescribed bounded region around the Stiefel manifold. A stepsize-based control of infeasibility also underlies the landing-type algorithms for smooth manifold optimization (\citet{ablin2022fast,ablin2024infeasible,gao2022optimization,schechtman2023orthogonal,vary2024optimization}).

Among the methods reviewed above, PG-SQP, ODCGM, and feasible
inexact manifold proximal-linear methods are particularly relevant.

\noindent\textbf{Relation to PG-SQP.}
The method of \citet{dai2025proximal} is closely
related to FSIPL. Introducing an outer variable $z$ rewrites
problem~\eqref{prob:primal} as minimizing $f(x)+g(z)$ subject to
$h(x)=0$ and $\mathcal A(x)-z=0$, so the presence of a nonlinear
mapping $\mathcal A$ alone does not distinguish the scope of the
two approaches. The lifted PG-SQP model and our subproblem share
regularized linearization ingredients. However, PG-SQP constructs
its tangential model after computing a normal displacement,
whereas FSIPL computes the tangential direction at $x^k$ and
evaluates the normal correction at $y^k=x^k+\eta_kd^k$.
FSIPL maintains only $x^k$ as an outer primal variable:
the variable $v$ in \eqref{exact-d-2} is used to solve the current subproblem
and is not an independently updated outer lifting variable.
Correspondingly, the constraint penalty in $\Phi_\alpha$
involves only $h$. Both analyses use uniform nondegeneracy of
the relevant constraint Jacobian for the feasibility and
stationarity guarantees discussed here. In FSIPL, this regularity
is specified on $\mathcal M_\theta$, the safeguard enforces
$\mathcal M_\kappa$, and level boundedness on
$\mathcal M_\theta$ yields primal boundedness.

\medskip
\noindent\textbf{Relation to ODCGM.}
ODCGM (\citet{schechtman2023orthogonal}) already treats smooth objectives
with general nonlinear equality constraints, combines tangential
and normal directions, and provides safe-stepsize safeguards
together with a deterministic $O(\varepsilon^{-2})$ complexity
bound. Its safeguard also admits a line-search variant.
The distinction of FSIPL is therefore not the use of infeasible
iterates or backtracking by itself. FSIPL treats the nonsmooth
composite term through an inexact proximal-linear model and
allows a nonzero tangency residual. This residual produces a
first-order term $\eta_k\Delta_k$ in the constraint-violation
estimate, in addition to the quadratic linearization error.
The correction is evaluated at the displaced point $y^k$,
and the backtracking reduces the tangential stepsize faster
than the normal stepsize. Consequently, their ratio decreases
until the normal contraction can absorb the violation generated
by the inexact tangential step. Propositions~\ref{Prop:hat_x,y}--\ref{prop:eta_min,tau_min-comp} establish
this compatibility and a uniform backtracking bound.

\medskip
\noindent\textbf{Relation to feasible inexact manifold proximal-linear methods.}
\citet{he2026inexact} and \citet{zheng2025new} allow inexact solution of
proximal-linear subproblems while keeping the computed
directions in the current tangent space, after which a retraction returns the next iterate to the manifold.
In particular, the adaptive accuracy conditions of
\citet{zheng2025new} control the subproblem objective error relative
to model reduction or step length while retaining the
tangent-space constraint. FSIPL allows inexactness in a
different part of the construction: neither
$\nabla h(x^k)d^k=0$ nor
$\nabla\mathcal A(x^k)d^k=v^k$ is required to hold exactly.
The dual stopping rule \eqref{inexact-solution-comp} instead gives the two residual
bounds in \eqref{ineq-res-lem}. These residuals enter the normal-correction
analysis and the merit-function descent estimate explicitly.
The contribution is thus the treatment of simultaneous primal
infeasibility and inexact linearized constraints without
manifold projection or retraction, rather than an improvement
of the common $O(\varepsilon^{-2})$ outer complexity exponent.
Our complexity result is stated using the shifted-composite
stationarity measure in Definition~\ref{def:optimization-stationary}.

\medskip
These comparisons motivate the following design question:
can a proximal-linear method preserve the proximal structure
of $g$, allow nonzero residuals in the linearized constraints,
and avoid manifold projection or retraction, while still
keeping accepted iterates in a prescribed regularity region
and retaining uniform stepsize and convergence guarantees?
The difficulty is the interaction of these requirements.
An inexact tangential solve can increase constraint violation,
a normal correction can increase the objective, and merit
descent alone does not enforce membership in the region on
which the regularity estimates are available. FSIPL addresses
this interaction by coupling a verifiable dual residual budget
with a normal correction and a two-stepsize safeguard.

\subsection{Contributions}
We propose the feasibility-safeguarded inexact proximal
linearized method (FSIPL) for problem~\eqref{prob:primal}. The method uses
the equality-defined geometry of the constraint set and
the proximal structure of $g$, without a projection or
retraction onto $\mathcal M$ in its iterative procedure.
Our contributions are threefold.

\begin{itemize}

\item \textbf{A verifiable inexact proximal-linear step.}
We compute an approximate primal--dual solution of the
strongly convex subproblem through its smooth unconstrained
dual. The additional quadratic regularization of the
composite displacement makes the split subproblem strongly
convex jointly in $(d,v)$ and yields a directly computable
dual gradient. A single gradient-norm stopping rule controls
both the tangency residual
$\|\nabla h(x^k)d^k\|$ and the consistency residual
$\|\nabla\mathcal A(x^k)d^k-v^k\|$.
The analysis incorporates these residuals without requiring
the returned direction to satisfy either of the two linearized constraints
exactly. The error tolerance is aligned with the
nonmonotone descent budget through
$\Delta_k\le c_\Delta\rho_k$.

\item \textbf{A coupled safeguard with uniform stepsize
and complexity guarantees.}
A normal correction at the tangential trial point is combined
with a backtracking rule that reduces the tangential stepsize
faster than the normal stepsize. This controls not only their
magnitudes but also their ratio, allowing the normal
contraction to absorb the constraint violation caused by
the inexact tangential step. Under Assumption~\ref{assp:N} and the
stated parameter conditions, Propositions~\ref{Prop:hat_x,y}--\ref{prop:eta_min,tau_min-comp} establish
invariance of the accepted-iterate region
$\mathcal M_\kappa$, a uniform bound on the number of
backtracking reductions, and uniform positive lower bounds
on the accepted stepsizes. Together with the dual estimates
and level boundedness on $\mathcal M_\theta$, these results
give boundedness of the primal and multiplier sequences,
stationarity of all cluster points, and an
$O(\varepsilon^{-2})$ outer iteration bound
(Theorems~\ref{the:complexity-comp} and~\ref{the:convger-comp}). Compactness of $\mathcal M$ and a
separate global distance-to-manifold error bound are not
assumed.

\item \textbf{An auxiliary-function analysis of
full-sequence convergence.}
The proximal optimality system involves a subgradient of $g$
at a shifted composite argument, while the merit function
uses the unshifted argument and a nonsmooth feasibility
penalty. To reconcile these structures, we construct an
extended-variable function incorporating the shifted
argument, graph representations of the residuals, and
memory terms. Under the parameter regime in Section~\ref{sec:converge},
which guarantees acceptance of full tangential steps,
we verify sufficient descent, relative error, and continuity
on the cluster set for this function. Its KL property then
yields finite length of the primal sequence and convergence
of both $\{x^k\}$ and $\{y^k\}$ to the same stationary point
(Theorem~\ref{conv-KL}).

\end{itemize}

Numerical experiments on sparse principal component analysis
and sparse spectral clustering illustrate the performance
of the method for identity and nonlinear composite mappings,
respectively.

\subsection{Organization}
The rest of the paper is organized as follows. Section~\ref{sec-Pre}  introduces notation and preliminary concepts. Section~\ref{sec:FSIPL} presents the FSIPL method and establishes its subsequential convergence and iteration complexity. Global sequential convergence under the Kurdyka--\L{}ojasiewicz property is proved in Section~\ref{sec:converge}. Section~\ref{sec:numerical} reports numerical experiments on sparse principal component analysis and sparse spectral clustering.

\section{Preliminaries}\label{sec-Pre}
 In this section, we recall some basic notation and preliminary results which will be used in this paper.  Let 
 $\mathbb{R}^n$ denote an $n$-dimensional Euclidean space with inner product $\langle x, y\rangle:=x^{\top}y$ for all $x,~y\in\mathbb{R}^n$, and the induced norm is denoted by $\left\|\cdot\right\|$. For linear operator \mbox{$\mathcal{T}:\mathbb{R}^n\to\mathbb{R}^m$}, let $\|\mathcal{T}\|_{\mathrm{op}}:=\sup_{x\neq 0}\frac{\|\mathcal{T}x\|}{\|x\|}$. The inner product on matrix spaces is given by the Frobenius inner product $\langle \cdot, \cdot \rangle_{F}$, i.e., $\langle A,B\rangle_F:= \operatorname{Tr}(A^\top B)$ for any matrices $A,B$ of compatible dimensions. Let $\mathbb{R}^n_{+}:=\{x:=(x_1,\dots,x_n)\in\mathbb{R}^n:x_i\geq0,~i=1,\dots,n\}$ and $\mathbb{R}^n_{++}:=\{x:=(x_1,\dots,x_n)\in\mathbb{R}^n:x_i>0,~i=1,\dots,n\}$. Let $\mathbb{S}^n$ denote the set of symmetric matrices in $\mathbb{R}^{n\times n}$. Given a point $x\in\mathbb{R}^n$ and a closed set $C\subseteq \mathbb{R}^n$, let  $\mathrm{dist}(x,C):=\inf_{y\in C}\|y-x\|$, $\mathrm{Proj}_{C}(x):=\argmin_{y\in C}\|y-x\|$, and $\delta_C$ be the indicator function associated with $C$.  For a smooth function $\psi:\mathbb{R}^n\to\mathbb{R}^m$, we write $\psi=[\psi_1,\dots,\psi_m]^{\top}$ and $\nabla \psi(x)=[\nabla \psi_1(x),\dots,\nabla \psi_m(x)]^{\top}\in\mathbb{R}^{m\times n}$. For given $t>0$, the Moreau envelope of a function $\phi:\mathbb{R}^n\to\mathbb{R}$ is defined by $M^{t}_\phi(x):=\min_{y\in\mathbb{R}^n} \big\{\phi(y)+\frac{1}{2t}\|y-x\|^2\big\}$, and the proximal operator is defined by $\mathrm{prox}_{t\phi}(x):=\operatorname{argmin}_{y\in\mathbb{R}^n} \big\{\phi(y)+\frac{1}{2t}\|y-x\|^2\big\}$. If $\phi$ is a closed convex function, then $\nabla M^{t}_\phi(x)=\frac{1}{t}(x-\mathrm{prox}_{t\phi}(x))$. For a function  $\psi: \mathbb{R}^n \to (-\infty, \infty]$, denote $\mathrm{dom}\,\psi:=\{x\in\mathbb{R}^n:\psi(x)<\infty\}$.
\begin{definition}\label{def:subdiff}
     Let $ \psi: \mathbb{R}^d \to (-\infty, \infty] $ be a proper and lower semicontinuous (lsc) function, and $ z \in \text{dom}\,\psi $.  
\begin{itemize}
    \item[(i)] The Fréchet (regular) subdifferential of $ \psi $ at $ z $, denoted by $ \hat{\partial} \psi(z) $, is the set of all vectors $ v \in \mathbb{R}^d $ satisfying  
\[
\psi(x) \geq \psi(z) + \langle v, x - z \rangle + o(\|x - z\|).
\]  

\item[(ii)] The (limiting) subdifferential of $ \psi $ at $ z $, denoted by $ \partial \psi(z) $, is the set of all vectors $ v \in \mathbb{R}^d $ such that there exist sequences $ \{z^k\}_{k \in \mathbb{N}} $ and $ \{v^k\}_{k \in \mathbb{N}} $, where $ z^k \to z $, $ \psi(z^k) \to \psi(z) $, $ v^k \in \hat{\partial} \psi(z^k) $, and $ v^k \to v $.  
\end{itemize}
For $ z \notin \text{dom}\,\psi $, we set $ \hat{\partial} \psi(z) := \partial \psi(z) := \emptyset $.  
\end{definition}
It follows that if $ \{z^k\}_{k \in \mathbb{N}} $ and $ \{v^k\}_{k \in \mathbb{N}} $ are such that $ z^k \to z $, $ \psi(z^k) \to \psi(z) $, $ v^k \in {\partial} \psi(z^k) $, and $ v^k \to v $, then $v\in\partial \psi (z)$.   When $ \psi $ is convex, the limiting subdifferential reduces to the classical subdifferential of the convex function $ \psi $. Furthermore, we also have the following useful subdifferential rules (\citet[Exercise~8.8(c),~Corollary~8.11,~Proposition~8.12,~Theorem~10.6]{rockafellar2009variational}): 
\begin{align*}
    \partial (\psi + \phi)(x) &= \partial \psi(x) + \nabla \phi(x),~\forall x \in \mathbb{R}^d,\\
    \partial\left(\psi\big(\mathcal{F}(y)\big)\right)&=\nabla\mathcal{F}(y)^{\top}\partial\psi(\mathcal{F}(y)), ~ \forall y \in \mathbb{R}^r,
\end{align*}
where $\psi:\mathbb{R}^d\to\mathbb{R}$ is proper closed convex, and $\phi:\mathbb{R}^{d}\to\mathbb{R}$, $\mathcal{F}:\mathbb{R}^{r}\to\mathbb{R}^d$ are continuously differentiable. Note that for $ x \notin \text{dom}\,\psi $, $\partial \psi(x)=\emptyset$, and it follows that $ \partial (\psi + \phi)(x)=\emptyset + \nabla \phi(x) = \emptyset $. 

If $x^*$ is a local minimizer of problem \eqref{prob:primal} under Assumption \ref{assp:N}, then the following KKT condition holds (\citet[Example~6.8,~Theorem~8.15,~Example~10.8]{rockafellar2009variational}): there exists $\lambda^*\in \mathbb{R}^m$ such that
\begin{equation}\label{KKT}
    \begin{cases}
0\in
\nabla f(x^*)+\nabla \mathcal{A}(x^*)^\top\partial g(\mathcal A(x^*))
+\nabla h(x^*)^\top\lambda^*,\\
h(x^*)=0.
\end{cases}
\end{equation}

This naturally leads to the following definition of optimization stationarity.
\begin{definition}\label{def:optimization-stationary}
(Optimization stationarity).
\par ~~
\begin{enumerate}
    \item[(i)] We say $x^*$ is a stationary point of problem \eqref{prob:primal} if there exists
\(\lambda^*\in\mathbb R^m\) such that the KKT condition~\eqref{KKT} holds.
    \item[(ii)] Let $\varepsilon>0$ be given and define
    \begin{equation*}
    Res(x,v,\lambda):=\max\{\mathrm{dist}(0,\nabla f(x)+\nabla\mathcal{A}(x)^{\top}\partial g(\mathcal{A}(x)+v)+ \nabla h(x)^{\top}\lambda),~\|v\|,~ \|h(x)\|\}.
\end{equation*} We say that $x$ is an $\varepsilon$-stationary point of problem \eqref{prob:primal} if there exist $v\in\mathbb{R}^p$ and  $\lambda\in \mathbb{R}^m$  such that
\begin{equation}\label{eps-KKT}
    Res(x,v,\lambda)<\varepsilon.
\end{equation}
\end{enumerate}
\end{definition}
The notion in item (ii) is an approximate counterpart of item (i). Indeed, the term $\|h(x)\|$ measures the feasibility violation, the term $\|v\|$ measures the discrepancy between the shifted composite argument $\mathcal{A}(x)+v$ and the original one $\mathcal{A}(x)$, and the distance term measures violation of the stationarity inclusion. In particular, if $Res(x,v,\lambda)=0$ for some $\lambda\in\mathbb{R}^m$, then $h(x)=0$ and the KKT condition \eqref{KKT} holds, which implies that $x$ is a stationary point of problem \eqref{prob:primal}.

\section{The proposed FSIPL algorithm and convergence analysis}\label{sec:FSIPL}
In this section, we first propose the feasibility-safeguarded inexact proximal linearized method (FSIPL) for solving problem \eqref{prob:primal} in Section \ref{subsec-FSIPL procedure}. Then we establish its outer iteration complexity for finding an $\varepsilon$-stationary point in Section \ref{subsec-sub-convergence}, and prove its subsequential convergence to a stationary point of problem \eqref{prob:primal}. 

\subsection{The FSIPL algorithm}\label{subsec-FSIPL procedure}
Each iteration of the proposed FSIPL method mainly consists of three stages. Let $x^k$ be the $k$-th iterate (possibly $x^k \notin \mathcal{M}$)  and define $\mathcal{M}^k := \left\{ x \in \mathbb{R}^n : h(x) = h(x^k) \right\}$ and $T_{x^k}\mathcal{M}^k:=\{d\in\mathbb{R}^n:\nabla h(x^k)d=0\}$. Clearly, $\mathcal{M}^k$ is also an embedded submanifold in $\mathbb{R}^n$ when $x^k$ is sufficiently close to $\mathcal{M}$. At the first stage of the $k$-th iteration, we (inexactly) solve the following strongly convex problem at every iteration to obtain an approximate tangential search direction $d^k$. The tangential subproblem is

\begin{equation}\label{exact-d}
    \min_{d \in T_{x^k}\mathcal{M}^k} \langle \nabla f(x^k), d \rangle + g\left(\mathcal{A}(x^k)+\nabla\mathcal{A}(x^k)d\right)  + \frac{1}{2t_k} (\| d \|^2 + \|\nabla \mathcal{A}(x^k)d\|^{2}),
\end{equation}
where  $0<\underline{t}\leq t_k\leq \overline{t}$ is a proximal parameter with $\underline{t}$ and $\overline{t}$ being predetermined. To address \eqref{exact-d}, we first reformulate it into 
\begin{equation}\label{exact-d-2}
    \begin{aligned}
       \min~~&\langle \nabla f(x^k), d\rangle
    +g(\mathcal{A}(x^k)+v)+\frac{1}{2t_k}\|d\|^2+\frac{1}{2t_k}\|v\|^2\\
    \text{~s.t.~}~~&{d\in T_{x^k}\mathcal{M}^k,~v=\nabla\mathcal{A}(x^k)d}.
\end{aligned}
\end{equation}
It is not hard to obtain the dual problem of \eqref{exact-d-2} (in a minimization form with constants omitted) :
\begin{equation}\label{prob:dual-comp}
    \begin{aligned}
\min_{(\lambda,~\mu)\in\mathbb{R}^m\times \mathbb{R}^p}~\{{G}_k(\lambda,\mu):=&\frac{t_k}{2}\|\nabla f(x^k)+\nabla h(x^k)^{\top}\lambda+\nabla \mathcal{A}(x^k)^{\top}\mu\|^2\\
&+\frac{t_k}{2}\|\mu\|^2-M^{t_k}_g(\mathcal{A}(x^k)+t_k\mu)\}.
\end{aligned}
\end{equation}
Invoking the property of the Moreau envelope (\citet[Theorem~6.60]{Beck2017First}), we see that $G_k$ is convex and continuously differentiable with a Lipschitz continuous gradient:
\begin{equation}\label{GG-comp-0}
\begin{aligned}
     \nabla_{\lambda}{G}_k(\lambda,\mu)=&
         t_k\nabla h(x^k)(\nabla f(x^k)+\nabla h(x^k)^{\top}\lambda+\nabla \mathcal{A}(x^k)^{\top}\mu),\\
        \nabla_{\mu}{G}_k(\lambda,\mu)=&t_k\nabla\mathcal{A}(x^k)(\nabla f(x^k)+\nabla h(x^k)^{\top}\lambda+\nabla \mathcal{A}(x^k)^{\top}\mu)\\
        &+\mathrm{prox}_{t_kg}(\mathcal{A}(x^k)+t_k\mu)-\mathcal{A}(x^k).
 \end{aligned}
 \end{equation}
Although one can recover the unique optimal solution to \eqref{exact-d-2} from those to the dual problem, in practice, it may be computationally expensive to exactly solve \eqref{prob:dual-comp}. Therefore, we propose to inexactly solve \eqref{prob:dual-comp}, which admits an approximate solution $(\lambda^k, \mu^k)$ to \eqref{prob:dual-comp} satisfying a verifiable inexact condition 
\begin{equation}\label{inexact-solution-comp}
    \|\nabla G_k(\lambda^k,\mu^k)\|\leq \Delta_k,
\end{equation}
where $0\leq \Delta_k\leq\overline{\Delta}$ with $\overline{\Delta}>0$ being predetermined. The case $\Delta_k=0$ is understood as an exact subproblem solve. For numerical inner solvers, a positive tolerance is used unless an exact solution is available.

Inspired by the relation between optimal solution to \eqref{exact-d-2} and \eqref{prob:dual-comp}, the approximate tangential search direction $d^k$ is evaluated by 
\begin{equation}\label{def:d^k-comp}
d^k=-t_k(\nabla f(x^k)+\nabla h(x^k)^{\top}\lambda^k+\nabla \mathcal{A}(x^k)^{\top}\mu^k),
\end{equation}
and we move the current iterate $x^k$ along $d^k$ with some stepsize $\eta_k>0$ to obtain an intermediate iterate $y^k$, i.e.,
\begin{equation}\label{step2}
    y^k:=x^k+\eta_k d^k.
\end{equation}
Next, at the second stage, we reduce the feasibility violation through a
normal correction step.
Let $N(x):=\frac{1}{2}\|h(x)\|^2$. We generate the trial iterate $\hat{x}^{k+1}$ by 
\begin{equation}\label{update-x}
        \hat{x}^{k+1}=y^k-\tau_k\nabla N(y^k)
    \end{equation} 
where $\tau_k>0$ is a stepsize.

Finally, at the third stage, we determine whether the trial iterate $\hat{x}^{k+1}$ is accepted by utilizing a merit function $\Phi_{\alpha}:\mathbb{R}^n\to\mathbb{R}$, which for a parameter $\alpha>0$ is defined as
\begin{equation}
    \Phi_{\alpha}(x):=F(x)+\alpha\|h(x)\|.
\end{equation}
Specifically, let $\{\rho_k:k\in\mathbb{N}\}$ be a nonnegative and summable scalar sequence, i.e.,  \begin{equation}\label{def:seq}
    \{\rho_k\}\in \mathcal{S}:=\big\{\{\rho_k:k\in\mathbb{N}\} :  \sum_{k=0}^{\infty}\rho_k<+\infty,~\rho_k\geq 0, k=0,\dots\big\},
\end{equation}
and let
\begin{equation}\label{def:v^k-comp}
    v^k:=\mathrm{prox}_{t_kg}(\mathcal{A}(x^k)+t_k\mu^k)-\mathcal{A}(x^k).
\end{equation} 
If $\|h(\hat{x}^{k+1})\|\leq \kappa$ and
\begin{equation}\label{Desc:Phi-comp-L}
      \begin{aligned}
         {\Phi}_{\alpha}(\hat{x}^{k+1})\leq &{\Phi}_{\alpha}(x^k)-\frac{\sigma}{2}(\eta_{k}^2\|d^{k}\|^2+\tau_k^2\|h(y^k)\|)\\
         &-\frac{\eta_k}{2t_k}\|v^k\|^2+\rho_k,
    \end{aligned}
  \end{equation}
the $\hat{x}^{k+1}$ is accepted and we set $x^{k+1}=\hat{x}^{k+1}$; otherwise, we decrease the stepsizes $\eta_k$ and $\tau_k$, and then re-compute $y^k$ by \eqref{step2} and accordingly generate a new trial iterate $\hat{x}^{k+1}$ by \eqref{update-x}. The feasibility acceptance condition $\|h(\hat{x}^{k+1})\|\leq\kappa$ ensures that every accepted iterate lies in $\mathcal{M}_{\kappa}\subset\mathcal{M}_{\theta}$. Together
with the nonmonotone descent condition \eqref{Desc:Phi-comp-L}, it
constitutes a backtracking line-search criterion for determining the stepsizes
$\eta_k$ and $\tau_k$.
 We shall prove in Proposition \ref{prop:Desc:Phi-comp} that this criterion must be fulfilled within finite inner iterations if $\alpha$ satisfies \eqref{def-alpha}. Moreover, as it will be seen in Section \ref{subsec-sub-convergence}, the feasibility safeguards guarantee that $\{x^k:k\in\mathbb N\}\subseteq\mathcal{M}_{\kappa}\subset\mathcal{M}_{\theta}$, while the descent estimate gives a uniform upper bound on $\{F(x^k)\}$. Assumption~\ref{assp:N}(vi) therefore yields the boundedness required for subsequential convergence.

Now, we are ready to formally present the complete framework of the FSIPL method in Algorithm~\ref{FSIPL-Line search-comp}. 

\begin{algorithm}[!ht]
 	\caption{Feasibility-safeguarded Inexact Proximal Linearized (FSIPL) Method }
     \label{FSIPL-Line search-comp}
    \begin{algorithmic}[1] 
        \REQUIRE $x_0\in \mathcal{M}$,  $\gamma\in(0,1)$, $\alpha >0$, $0<\kappa<\theta$, $(\underline{t},\overline{t},\overline{\Delta},\overline{\eta},\overline{\tau},\theta,\sigma,c_{\Delta})\in\mathbb{R}^8_{++}$, $\{\rho_k:k\in\mathbb{N}\}\in \mathcal{S}$  ; 
         \FOR{$k=0,1,2, \dots$}
         \STATE Choose $t_k\in[\underline{t},\overline{t}]$ and $\Delta_k\in[0,\min\{\overline{\Delta},c_{\Delta}\rho_k\}]$.
         \STATE Approximately solve \eqref{prob:dual-comp} to obtain $({\lambda}^k,\mu^k)$ such that the inexact condition \eqref{inexact-solution-comp} holds.
         \STATE Compute $d^k$ by \eqref{def:d^k-comp} and $v^k$ by \eqref{def:v^k-comp}
          \STATE Set $\eta_k\leftarrow\overline{\eta}$ and $\tau_k\leftarrow\overline{\tau}$.
        \STATE Compute  $y^k$ via \eqref{step2}. 
          \STATE Compute $\hat{x}^{k+1}$ via \eqref{update-x}.  
         \IF{$\hat{x}^{k+1}$ satisfies \eqref{Desc:Phi-comp-L} and $\|h(\hat{x}^{k+1})\|\leq \kappa$, }
         \STATE set $x^{k+1}\leftarrow\hat{x}^{k+1}$,
         \ELSE
         \STATE  set $\eta_k\leftarrow \frac{\gamma}{2}\eta_k$ and $\tau_k\leftarrow \gamma \tau_k$, 
         \STATE go to \textbf{Line~6}.    
         \ENDIF
         \ENDFOR         
    \end{algorithmic}
\end{algorithm}

To end this subsection, we remark that in the first stage of FSIPL, an approximate solution $(\lambda^k,\mu^k)$ satisfying \eqref{inexact-solution-comp}, as well as $(d^k,v^k)$ defined in \eqref{def:d^k-comp} and \eqref{def:v^k-comp}, can be computed in a simple way in the case of $\mathcal{A} = I_n$. Specifically, we consider the following strongly convex subproblem:
\begin{equation}\label{primal-A=I}
       \min_{\substack{d \in T_{x^k}\mathcal{M}^k}}\langle \nabla f(x^k), d\rangle+g(x^k+d)+\frac{1}{t_k}\|d\|^2,
\end{equation}
and its dual problem (in a minimization form with constants omitted) 
\begin{equation}\label{dual-A=I}
    \begin{aligned}
    \min_{\lambda\in\mathbb{R}^m} \widetilde{G}_k(\lambda):=&t_k\|\frac{1}{2}(\nabla f(x^k)+\nabla h(x^k)^{\top}\lambda)\|^2\\&-M^{t_k/2}_g\left(x^k-\frac{t_k}{2}(\nabla f(x^k)+\nabla h(x^k)^{\top}\lambda)\right).
\end{aligned}
\end{equation}
We inexactly solve \eqref{dual-A=I} to find $\lambda^k\in\mathbb{R}^m$ such that $\|\nabla \widetilde{G}_k(\lambda^k)\|\leq\Delta_k$. By the relation between optimal solutions of \eqref{primal-A=I} and \eqref{dual-A=I}, we set
\begin{equation}\label{def-d^k-A=I}
   d^k=\mathrm{prox}_{\frac{t_k}{2}g}\left(x^k-\frac{t_k}{2}(\nabla f(x^k)+\nabla h(x^k)^{\top}\lambda^k)\right)-x^k.
\end{equation}
Accordingly, we set $v^k=d^k$ and 
\begin{equation}\label{def-mu-A=I}
    \mu^k=-\frac{1}{t_k}d^k-(\nabla f(x^k)+\nabla h(x^k)^{\top}\lambda^k).
\end{equation}
As shown in Appendix \ref{append-2}, the iterates $\lambda^k,~d^k,~v^k,~\mu^k$ evaluated as above are consistent with \eqref{def:d^k-comp}, \eqref{def:v^k-comp}, and satisfy the inexactness criterion \eqref{inexact-solution-comp}.

We conclude this subsection by explaining the role of the
coupled backtracking rule. For a trial tangential stepsize
$\eta>0$, the inexactness certificate gives
\[
\|h(x^k+\eta d^k)\|
\le
\|h(x^k)\|+\eta\Delta_k
+\frac{L_h}{2}\eta^2\|d^k\|^2.
\]
Thus, unlike an exactly tangential step, the computed direction
may generate a first-order constraint error. The normal
correction must absorb this error as well as the quadratic
linearization error, while satisfying the merit-function
acceptance test.

After $j$ rejected trials, Algorithm~\ref{FSIPL-Line search-comp} uses
\[
\eta^{(j)}=\bar\eta(\gamma/2)^j,\qquad
\tau^{(j)}=\bar\tau\gamma^j,\qquad
\frac{\eta^{(j)}}{\tau^{(j)}}
=\frac{\bar\eta}{\bar\tau}2^{-j}.
\]
Consequently, both stepsizes and their ratio decrease.
This distinction is important: decreasing both stepsizes
by the same factor would leave their ratio unchanged and
would not, by itself, enforce the sufficient coupled
condition
$\eta D(t_k,\Delta_k)\le\tau C_1\kappa/2$
used in Proposition~\ref{Prop:hat_x,y}. Our rule eventually satisfies this
condition together with the remaining sufficient
acceptance conditions. Proposition~\ref{prop:eta_min,tau_min-comp} makes this argument
uniform over the outer iterations and obtains positive
lower bounds on the accepted stepsizes.

The algorithm checks the corrected point through the
constraint-residual and merit-function tests; it does not
evaluate a projection onto $\mathcal M$ or the distance to
$\mathcal M$. The initial trial stepsizes need not be set
to the conservative sufficient bounds derived in the
analysis. This does not mean that tangential and normal
progress are uncoupled: their interaction is controlled
explicitly by the unequal backtracking factors and the
error budget $\Delta_k\le c_\Delta\rho_k$.

\subsection{Subsequential convergence and iteration complexity}\label{subsec-sub-convergence}
This subsection is devoted to the subsequential convergence and iteration complexity of the proposed FSIPL method. We first establish some basic properties regarding \textbf{Lines~2--7} in a single iteration of Algorithm \ref{FSIPL-Line search-comp} in Lemma \ref{lemma:d^k-bound-comp}, \eqref{x,y_range-comp} in Proposition \ref{Prop:hat_x,y}, and Proposition \ref{prop:Desc:Phi-comp}. With the help of these results, we then show the well-definedness of the proposed method in Proposition \ref{prop:eta_min,tau_min-comp}. Finally, we establish the iteration complexity and subsequential convergence in Theorems \ref{the:complexity-comp} and \ref{the:convger-comp} respectively.

Before formally proving the convergence results, we record several direct consequences of Assumption \ref{assp:N}. First, the eigenvalue bound in item (i) of Assumption \ref{assp:N} yields that $\sqrt{C_1} \|\lambda\| \leq \| \nabla h(x)^{\top} \lambda \| \leq \sqrt{C_2} \|\lambda\|$, for all $x \in \mathcal{M}_{\theta}$ and $\lambda \in \mathbb{R}^m$.  Finally, item (iii) indicates that $\|z\| \leq \ell_g$ for all $z \in \partial g(y)$ and $y \in \mathbb{R}^p$. These estimates will be used repeatedly below without further mention. Now we begin our convergence analysis. The following lemma gives the bound estimate of the sequence $\{(d^k,v^k,\lambda^k,\mu^k) :  k\in\mathbb{N}\}$. 
\begin{lemma}\label{lemma:d^k-bound-comp}
   Assume that $x^k\in\mathcal{M}_{\theta}$. Let  $(\lambda^k,\mu^k)\in\mathbb{R}^m\times\mathbb{R}^p$ satisfy \eqref{inexact-solution-comp},  and let $d^k$ and $v^k$ be defined in \eqref{def:d^k-comp} and \eqref{def:v^k-comp} respectively.  Then, the following inequalities hold: 
 \begingroup
\setlength{\abovedisplayskip}{3pt plus 1pt minus 1pt}
\setlength{\belowdisplayskip}{3pt plus 1pt minus 1pt}
\setlength{\abovedisplayshortskip}{0pt plus 1pt}
\setlength{\belowdisplayshortskip}{3pt plus 1pt minus 1pt}
   \begin{align*}
        \|d^k\|&\leq d(t_k,\Delta_k):=t_k(2\ell_f+\ell_g)+\left(\frac{\ell_{\mathcal{A}}+1}{\sqrt{C_1}}+1\right){\Delta}_{k}, \tag{i}\label{ineq:d^k-bound-comp}
        \\
        \|v^k\|&\leq t_k(\ell_f+2\ell_g)+\left(\frac{\ell_{\mathcal{A}}+1}{\sqrt{C_1}}+1\right)\Delta_k, \tag{ii}\label{ineq:v^k-bound-comp}\\
        \|\lambda^k\|&\leq \lambda(t_k,\Delta_k):= \frac{1}{\sqrt{C_1}}\left((3+\ell_{\mathcal{A}})\ell_f+(3\ell_{\mathcal{A}}+1)\ell_g+\left(\frac{\ell_{\mathcal{A}}+1}{\sqrt{C_1}}+1\right)\frac{(\ell_{\mathcal{A}}+1)\Delta_k}{t_k}\right),\tag{iii}\label{ineq:lambda^k-bound-comp}\\
         \|\mu^k\|&\leq\mu(t_k,\Delta_k):= \ell_f+3\ell_g+\left(\frac{\ell_{\mathcal{A}}+1}{\sqrt{C_1}}+1\right)\frac{\Delta_k}{t_k}.\tag{iv}\label{ineq:mu^k-bound-comp}
    \end{align*}  
    \endgroup
\end{lemma}
\begin{proof}
We first prove \eqref{ineq:d^k-bound-comp} and \eqref{ineq:v^k-bound-comp}.
Let \[e_{d}^k:=\nabla h(x^k)d^k,~~e^k_{v}:=\nabla \mathcal{A}(x^k)d^k-v^k.\] We derive from \eqref{GG-comp-0}, \eqref{inexact-solution-comp}, and the definitions of $d^k$ in \eqref{def:d^k-comp} and $v^k$ in \eqref{def:v^k-comp} that
\begin{equation}\label{ineq-res-lem}
   \|e_{d}^k\|\leq \Delta_k,  ~~\|e^k_{v}\|\leq \Delta_k.
\end{equation}
Furthermore, note that
\begin{equation}\label{d-v-min}
\begin{aligned}
    (d^{k},v^k)=\arg\min &\langle \nabla f(x^k),d\rangle+g(\mathcal{A}(x^k)+v)+\frac{1}{2t_k}\|d\|^2+\frac{1}{2t_k}\|v\|^2\\
    &+\langle \lambda^k, \nabla h(x^k)d\rangle+\langle\mu^k,\nabla\mathcal{A}(x^k)d-v\rangle,
\end{aligned}  
\end{equation}
by the convexity of $g$, we have
\begin{equation}\label{prf:d^k,v^k}
 \begin{aligned}
        (d^{k},v^k)=&\arg\min \langle \nabla f(x^k),d\rangle+g(\mathcal{A}(x^k)+v)+\frac{1}{2t_k}\|d\|^2+\frac{1}{2t_k}\|v\|^2.\\
        &\text{ s.t. } \nabla h(x^k)d=e_d^k,~~\nabla\mathcal{A}(x^k)d-v=e^k_v.
    \end{aligned}
    \end{equation}
Let \[\hat{d}^k=\nabla h(x^k)^{\top}H(x^k)^{-1}e^k_d,~~\hat{v}^k=\nabla\mathcal{A}(x^k)\hat{d}^k-e^k_v. \]
Then,   \[\nabla h(x^k)\hat{d}^k=\nabla h(x^k)d^k,~~\nabla\mathcal{A}(x^k)\hat{d}^k-\hat{v}^k=e^k_v.\] 
It follows from \eqref{prf:d^k,v^k} and the convexity of $g$ that
\begin{equation*}
    \begin{aligned}
       &g(\mathcal{A}(x^k)+{v}^k)+\langle \nabla f(x^k), {d}^k\rangle+\frac{1}{2t_k}\|{d}^k\|^2+\frac{1}{2t_k}\|{v}^k\|^2\\
       \leq& g(\mathcal{A}(x^k)+\hat{v}^k)+\langle \nabla f(x^k), \hat{d}^k\rangle+\frac{1}{2t_k}\|\hat{d}^k\|^2+\frac{1}{2t_k}\|\hat{v}^k\|^2.
    \end{aligned}
\end{equation*}

  Combining the Lipschitz continuity of $f$ and $g$, this leads to  
\begin{equation}
    \frac{1}{2t_k}(\|{d}^k\|^2+\|v^k\|^2)\leq \ell_f(\|d^k\|+\|\hat{d}^k\|)+\ell_g(\|v^k\|+\|\hat{v}^k\|)+\frac{1}{2t_k}(\|\hat{d}^k\|^2+\|\hat{v}^k\|^2),
\end{equation}
which is equivalent to 
\begin{equation}
    (\|{d}^k\|-t_k\ell_f)^2+(\|v^k\|-t_k\ell_g)^2\leq (\|\hat{d}^k\|+t_k\ell_f)^2+(\|\hat{v}^k\|+ t_k\ell_g)^2.
\end{equation}
This implies that
\begin{equation}\label{prf:dv-1}
\begin{aligned}
    \|d^k\|\leq \|\hat{d}^k\|+\|\hat{v}^k\|+t_k(2\ell_f+\ell_g),\\
    \|v^k\|\leq \|\hat{d}^k\|+\|\hat{v}^k\|+t_k(\ell_f+2\ell_g).
    \end{aligned}
\end{equation}
From Assumption \ref{assp:N}(i) and \eqref{ineq-res-lem}, we deduce that
\begin{equation}
\begin{aligned}
    \|\hat{d}^k\|&=\|\nabla h(x^k)^{\top}H(x^k)^{-1}e_d^k\|\leq\|\nabla h(x^k)^{\top}H(x^k)^{-1}\|_{\mathrm{op}}\|e_d^k\|\leq\frac{1}{\sqrt{C_1}}\Delta_k,\\
    \|\hat{v}^k\|&=\|\nabla\mathcal{A}(x^k)\hat{d}^k-e^k_v\|\leq\ell_{\mathcal{A}}\|\hat{d}^k\|+\|e^k_v\|\leq (\frac{\ell_{\mathcal{A}}}{\sqrt{C_1}}+1)\Delta_k,
    \end{aligned}
\end{equation}
which together with \eqref{prf:dv-1}  implies \eqref{ineq:d^k-bound-comp} and
\eqref{ineq:v^k-bound-comp}.

Next, we prove \eqref{ineq:lambda^k-bound-comp} and \eqref{ineq:mu^k-bound-comp}. We derive from \eqref{def:d^k-comp} and \eqref{def:v^k-comp} that 
    \begin{equation}\label{eq-optimal-d,v}
    \begin{aligned}
        0&= \nabla f(x^k)+\frac{1}{t_k}d^k+\nabla h(x^k)^{\top}\lambda^k+\nabla\mathcal{A}(x^k)^{\top}\mu^k,\\
        0&\in \partial g\big(\mathcal{A}(x^k)+v^k\big)+\frac{1}{t_k}v^k-\mu^k.
        \end{aligned}
    \end{equation}
 Consequently, we obtain 
   \begin{equation}\label{ineq-mu-prf}
   \|\mu^k\|\leq \ell_g+\frac{\|v^k\|}{t_k}\leq \ell_f+3\ell_g+(\frac{(\ell_{\mathcal{A}}+1)}{\sqrt{C_1}}+1)\frac{\Delta_k}{t_k},
\end{equation}
where the first inequality comes from \eqref{eq-optimal-d,v} and the second inequality comes from \eqref{ineq:v^k-bound-comp}. This implies that \eqref{ineq:mu^k-bound-comp} holds.
We also obtain
   \begin{equation}
       \sqrt{C_1}\|\lambda^k\|\leq\|\nabla h(x^k)^{\top}\lambda^k\|\leq \ell_f+\frac{\|d^k\|}{t_k}+\ell_{\mathcal{A}}\|\mu^k\|,\label{ineq:lambda-comp}
   \end{equation} 
 where the first inequality comes from Assumption \ref{assp:N}(i) and the second inequality comes from the Lipschitz continuity of $f$ and $\mathcal{A}$ and \eqref{eq-optimal-d,v}.
Then, \eqref{ineq:lambda^k-bound-comp} follows immediately from \eqref{ineq-mu-prf}, \eqref{ineq:lambda-comp} and \eqref{ineq:d^k-bound-comp}.
\end{proof}

Based on lemma \ref{lemma:d^k-bound-comp}, we have some further results on the iterates $y^k$ and $\hat{x}^{k+1}$.
\begin{proposition}\label{Prop:hat_x,y}
Let  $0<\kappa<\theta$ and  $x^k\in\mathcal{M}_{\kappa}$, and let  $(\lambda^k,\mu^k)$, $d^k$, $v^k$, $y^k$ and $\hat{x}^{k+1}$ be generated as \textbf{Lines~2--7} in Algorithm \ref{FSIPL-Line search-comp}. Assume that $\tau_k$ and $\eta_k$ satisfy
     \begin{equation}\label{ub:eta,tau-comp}
         \begin{aligned}
         &{\tau}_k\leq \hat{\tau}:=\min\left\{\frac{ C_1}{\theta L_h C_2},\frac{1}{C_2}\right\},\\
        &{\eta}_k\leq \min\left\{\hat{\eta}_1(t_k,\Delta_k),\frac{\tau_kC_1\kappa}{2D(t_k,\Delta_k)}\right\}, 
    \end{aligned} 
     \end{equation}  
     where 
     \begin{equation}\label{def-eta1-D}
   D(t_k,\Delta_k)=\Delta_k+ \frac{L_h}{2}d(t_k,\Delta_k)^2,\qquad\hat{\eta}_1(t_k,\Delta_k):=  \min\left\{\frac{\theta-\kappa}{D(t_k,\Delta_k)},1\right\}.
     \end{equation}
   Then, the following statements hold:  
   \begingroup
\setlength{\abovedisplayskip}{3pt plus 1pt minus 1pt}
\setlength{\belowdisplayskip}{3pt plus 1pt minus 1pt}
\setlength{\abovedisplayshortskip}{0pt plus 1pt}
\setlength{\belowdisplayshortskip}{3pt plus 1pt minus 1pt}
   \begin{align*}
    &y^k\in \mathcal{M}_{\theta} \text{ and } \hat{x}^{k+1}\in \mathcal{M}_{\kappa}.\tag{i}\label{x,y_range-comp}\\
      \tag{ii} \label{ineq:desc:F(x)-comp}    
        &F(\hat{x}^{k+1})-F(x^k)-(\ell_f+\ell_g\ell_{\mathcal{A}})\|\hat{x}^{k+1}-y^k\|\\
        \leq& -(\frac{1}{t_{k}\eta_{k}}-L_f-{\ell_gL_{\mathcal{A}}})\frac{\eta_{k}^2}{2}\|d^{k}\|^2-\frac{\eta_k}{2t_k}\|v^k\|^2+\eta_{k}(\ell_g+\|\lambda^{k}\|+\|\mu^{k}\|)\Delta_{k}.  \\ 
        &\|h(\hat{x}^{k+1})\|-(\eta_{k}\Delta_{k}+\frac{L_h}{2}\eta_{k}^2\|d^{k}\|^2)\leq\|h(x^k)\|-(C_1\tau_k\|h(y^k)\|-\frac{L_hC_2}{2}\tau_k^2\|h(y^k)\|^2).\tag{iii}\label{ineq:Desc:h(x)}
    \end{align*}
  \endgroup 
\end{proposition}
   \begin{proof}
   We first prove \eqref{x,y_range-comp}. Since  $\|h(x^k)\|\leq \kappa$, by Assumption \ref{assp:N}(i) and \eqref{GG-comp-0}-\eqref{step2}, we have 
   \begin{equation}\label{ineq-h(y^k)}
       \|h(y^k)\|\leq \|h(x^k)\|+\eta_k\Delta_k+\frac{\eta_k^2 L_h}{2}\|d^k\|^2\leq \theta,
   \end{equation}
    where the second inequality follows from Lemma \ref{lemma:d^k-bound-comp} \eqref{ineq:d^k-bound-comp} and the upper bound of $\eta_k$.
We next show that $\hat{x}^{k+1}\in\mathcal{M}_{\kappa}$. 
    In view of Assumption \ref{assp:N}(i), we obtain
    \begin{equation}\label{ineq-lemma-h(x)}
        \begin{aligned}
        \|h(\hat{x}^{k+1})\|&\leq \|h(y^k)+\nabla h(y^k)(\hat{x}^{k+1}-y^k)\|+\frac{L_h}{2}\tau_k^2\|\nabla N(y^k)\|^2\\
        &\leq \|h(y^k)\|-\tau_kC_1\|h(y^k)\|+\frac{L_hC_2}{2}\tau_k^2\|h(y^k)\|^2\\
        &\leq (1-\frac{\tau_kC_1}{2})\|h(y^k)\|,
    \end{aligned}
    \end{equation}
    where the second inequality comes from the fact that $\|H(y^k)\|_{\mathrm{op}}\leq C_2$ and  $\tau_k\leq\frac{1}{C_2}$, and the last inequality due to the upper bound of $\tau_k$. This and \eqref{ineq-h(y^k)} yield that  
    \begin{equation*}
    \begin{aligned}
        \|h(\hat{x}^{k+1})\|&\leq \|h(x^k)\|-\frac{\tau_kC_1}{2}\|h(y^k)\|+\eta_k\Delta_k+\frac{\eta_k^2 L_h}{2}\|d^k\|^2
    \end{aligned}
    \end{equation*}
    If $\kappa\leq\|h(y^k)\|\leq \theta$,
   by the upper bound  of $\eta_k$ controlled by $\tau_k$, we have $\|h(\hat{x}^{k+1})\|\leq \|h(x^k)\|\leq \kappa$. On the other hand, if $\|h(y^k)\|\leq \kappa$, \eqref{ineq-lemma-h(x)} and the upper bound  of $\tau_k$ yield that \[\|h(\hat{x}^{k+1})\|\leq \|h(y^k)\|\leq \kappa.\] 
Next, we prove \eqref{ineq:desc:F(x)-comp}. It follows from \eqref{d-v-min} that
\begin{equation}\label{ineq:lem1-prox-comp}
    \begin{aligned}
&g(\mathcal{A}(x^{k})+v^{k})+\langle \nabla f(x^{k}),d^{k}\rangle+\frac{1}{2t_k}\|d^{k}\|^2+\frac{1}{2t_k}\|v^{k}\|^2\\
       \leq& g\big(\mathcal{A}(x^{k})\big)-\langle \lambda^{k}, \nabla h(x^{k})d^{k}\rangle-\langle\mu^{k},\nabla\mathcal{A}(x^{k})d^{k}-v^{k}\rangle.
    \end{aligned}
\end{equation}
Since for all $s\in[0,1]$
 \[\|h(x^k+s(y^k-x^k))\|\leq \|h(x^k)\|+s\eta_k\Delta_k+\frac{s^2\eta_k^2 L_h}{2}\|d^k\|^2\leq \theta,\] we also derive from Assumption \ref{assp:N}(iii) and (iv) that
  \begin{equation}\label{ineq:lem1-g-comp}
      \begin{aligned}
        g(\mathcal{A}(x^{k}+\eta_{k} d^{k}))\leq& g(\mathcal{A}(x^{k})+\eta_{k}\nabla\mathcal{A}(x^{k}) d^{k})\\&+\ell_g\|\mathcal{A}(x^{k}+\eta_{k} d^{k})-\mathcal{A}(x^{k})-\eta_{k}\nabla\mathcal{A}(x^{k}) d^{k}\|\\
        \leq&g(\mathcal{A}(x^{k})+\eta_{k}v^{k})+\eta_{k}\ell_g\|\nabla\mathcal{A}(x^{k})d^{k}-v^{k}\|+\frac{\ell_gL_{\mathcal{A}}}{2}\eta_{k}^2\|d^{k}\|^2\\
        \leq&(1-\eta_{k})g\big(\mathcal{A}(x^{k})\big)+\eta_{k}g(\mathcal{A}(x^{k})+v^{k})\\&+\eta_{k}\ell_g\|\nabla\mathcal{A}(x^{k})d^{k}-v^{k}\|+\frac{\ell_gL_{\mathcal{A}}}{2}\eta_{k}^2\|d^{k}\|^2,
    \end{aligned}
  \end{equation}
 Here, the first inequality follows from the $\ell_g$-Lipschitz continuity of $g$, the second one follows from the $L_{\mathcal{A}}$-Lipschitz continuity of $\nabla\mathcal{A}$. and the convexity of $g$ yields the third inequality.   On the other hand, it follows from Assumption \ref{assp:N}(ii)  that \begin{equation}\label{ineq:lem1-f-comp}
\begin{aligned}
        f(y^{k})\leq f(x^{k})+\langle \nabla f(x^{k}), y^{k}-x^{k}\rangle +\frac{L_f}{2}\|y^{k}-x^{k}\|^2.
        \end{aligned}
    \end{equation}
    Summing \eqref{ineq:lem1-prox-comp} multiplied by $\eta_k$  with \eqref{ineq:lem1-g-comp} and \eqref{ineq:lem1-f-comp}, and  by \eqref{ineq-res-lem}, we deduce that
    \begin{equation}\label{prf-F(y),F(x)-comp}
    \begin{aligned}
        F(y^{k})\leq &F(x^{k})-(\frac{1}{t_{k}\eta_{k}}-L_f-{\ell_gL_{\mathcal{A}}})\frac{\eta_{k}^2}{2}\|d^{k}\|^2-\frac{\eta_k}{2t_k}\|v^k\|^2+\eta_{k}(\ell_g+\|\lambda^{k}\|+\|\mu^{k}\|)\Delta_{k}.
         \end{aligned}
    \end{equation}
As $F$ is $(\ell_f+\ell_g\ell_{\mathcal{A}})$-Lipschitz continuous,
    \[F(\hat{x}^{k+1})\leq F(y^k)+(\ell_f+\ell_g\ell_{\mathcal{A}})\|\hat{x}^{k+1}-y^k\|.\]
    This together with \eqref{prf-F(y),F(x)-comp} implies \eqref{ineq:desc:F(x)-comp}.
    
    Finally, we prove \eqref{ineq:Desc:h(x)}. From Assumption \ref{assp:N}(i) and \eqref{ineq-res-lem},   we have 
    \begin{equation}\label{ineq:lemma-h(x)-1}
         \begin{aligned}
       \|h(y^k)\|&\leq \|h(x^k)+\eta_k\nabla h(x^k)d^k\|+\frac{L_h}{2}\eta_k^2\|d^k\|^2\\&\leq\|h(x^k)\|+\eta_k\Delta_k+\frac{L_h}{2}\eta_k^2\|d^k\|^2,
    \end{aligned} 
    \end{equation}
and
   \begin{equation}\label{ineq:h(x)-h(y)}
       \begin{aligned}
        \|h(\hat{x}^{k+1})\|&\leq \|h(y^k)+\nabla h(y^k)(\hat{x}^{k+1}-y^k)\|+\frac{L_h}{2}\|\hat{x}^{k+1}-y^k\|^2\\
        &\leq (1-\tau_kC_1)\|h(y^k)\|+\frac{L_hC_2}{2}\tau_k^2\|h(y^k)\|^2.
    \end{aligned}
   \end{equation}
 These lead to item \eqref{ineq:Desc:h(x)}. 
   We complete the proof.
  \end{proof}

Analogous to the safeguard mechanism developed in \citet{ablin2024infeasible}, Proposition~\ref{Prop:hat_x,y} provides explicit step-size bounds $\hat{\eta}_1$ and $\hat{\tau}$, which together with the relation between $\eta_k$ and $\tau_k$ guarantees the trial iterate $\hat{x}^{k+1}$ remains within the safeguard  region $\mathcal{M}_{\kappa}$. This indicates that a line-search procedure can be directly integrated into our algorithm to adaptively select step sizes ensuring all iterates stay in $\mathcal{M}_{\kappa}$ without imposing restrictive a priori step-size choices.



With the help of Proposition \ref{Prop:hat_x,y}, we obtain the next proposition regarding the merit function $\Phi_{\alpha}$. Recall from Lemma~\ref{lemma:d^k-bound-comp} that the multiplier bounds $\lambda(t,\Delta)$ and $\mu(t,\Delta)$ are nondecreasing in $\Delta$ and nonincreasing in $t$. Define the uniform constants
\begin{equation}\label{def:uniform-const}
\lambda_{\max}:=\lambda(\underline{t},\overline{\Delta}),\qquad \mu_{\max}:=\mu(\underline{t},\overline{\Delta}),\qquad C_e:=\alpha+\ell_g+\lambda_{\max}+\mu_{\max}.
\end{equation}
\begin{proposition}\label{prop:Desc:Phi-comp}
     Let  $0<\kappa<\theta$ and  $x^k\in\mathcal{M}_{\kappa}$, and let
     \begin{equation}\label{def-alpha}
         \alpha >\frac{M_1}{C_1},\qquad M_1:=\sqrt{C_2}(\ell_f+\ell_g\ell_{\mathcal{A}}).
     \end{equation}
     Let  $(\lambda^k,\mu^k)$, $d^k$, $v^k$, $y^k$ and $\hat{x}^{k+1}$ be generated as \textbf{Lines~2--7} in Algorithm \ref{FSIPL-Line search-comp}. Assume that $\eta_k$ and $\tau_k$ satisfy
   \begin{equation}\label{ub:eta,tau-comp-2}
         \begin{aligned}   
         {\tau}_k&\leq \hat{\tau}_2:= \min\left\{\hat{\tau}, \frac{2(\alpha C_1-M_1)}{\alpha L_hC_2\theta+\sigma}\right\},\\      
         \eta_k&\leq \min\left\{\hat{\eta}_2(t_k,\Delta_k),\frac{\tau_kC_1\kappa}{2D(t_k,\Delta_k)}\right\},
     \end{aligned}
     \end{equation} 
     where $L_{\mathrm{mod}}=L_f+\ell_gL_{\mathcal{A}}+\alpha L_h+\sigma$,  $D(t_k,\Delta_k)$ defined in \eqref{def-eta1-D} and
     \begin{equation}\label{def-eta2}
         \hat{\eta}_2(t_k,\Delta_k):=\min\left\{\hat{\eta}_1(t_k,\Delta_k), \frac{1}{t_kL_{\mathrm{mod}}},\frac{1}{c_{\Delta}C_e}\right\}
     \end{equation}
   with $\hat{\eta}_1$ defined in \eqref{def-eta1-D} and $C_e$  defined in \eqref{def:uniform-const}. Then,
  \begin{equation}\label{Desc:Phi-comp}
      \begin{aligned}
         {\Phi}_{\alpha}(\hat{x}^{k+1})\leq &{\Phi}_{\alpha}(x^k)-\frac{\sigma}{2}(\eta_{k}^2\|d^{k}\|^2+\tau_k^2\|h(y^k)\|)-\frac{\eta_k}{2t_k}\|v^k\|^2+\rho_k.
    \end{aligned}
  \end{equation}  
\end{proposition}
\begin{proof}
    By Assumption \ref{assp:N}(i) and Proposition \ref{Prop:hat_x,y}(i),  
     \begin{equation}\label{prf-(x-y)}
        \|\hat{x}^{k+1}-y^k\|=\tau_k\|\nabla h(y^k)^{\top} h(y^k)\|\leq \tau_k\sqrt{C_2}\|h(y^k)\|.
     \end{equation}  
   Combining \eqref{prf-(x-y)} with items \eqref{ineq:desc:F(x)-comp} and \eqref{ineq:Desc:h(x)} multiplied by $\alpha$ in Proposition \ref{Prop:hat_x,y},
    we have 
    \begin{equation}\label{prf:phi(x)-1}
        \begin{aligned}
        \Phi_{\alpha}(\hat{x}^{k+1})\leq& \Phi_{\alpha}(x^k)-(\frac{1}{t_k\eta_k}-L_f-\ell_gL_{\mathcal{A}}-\alpha L_h)\frac{\eta_k^2}{2}\|d^k\|^2-\frac{\eta_k}{2t_k}\|v^k\|^2\\
        &-(\frac{2(\alpha C_1-M_1)}{\tau_k}-\alpha  L_hC_2 \theta)\frac{\tau_k^2}{2}\|h(y^k)\|+\eta_k(\alpha+\ell_g+\|\lambda^k\|+\|\mu^k\|)\Delta_k.
    \end{aligned}
    \end{equation}
    Moreover, Lemma~\ref{lemma:d^k-bound-comp} together with $t_k\ge\underline t$, $\Delta_k\le\overline\Delta$, and \eqref{def:uniform-const} gives $\|\lambda^k\|\le\lambda_{\max}$ and $\|\mu^k\|\le\mu_{\max}$, so  the last term in \eqref{prf:phi(x)-1} satisfies
\[
\eta_k(\alpha+\ell_g+\|\lambda^k\|+\|\mu^k\|)\Delta_k
\le \eta_k C_e\Delta_k
\le \eta_k C_e c_\Delta\rho_k
\le \rho_k,
\]
where the last two inequalities use $\Delta_k\le c_\Delta\rho_k$ and $\eta_k\le 1/(c_\Delta C_e)$ from \eqref{ub:eta,tau-comp-2}--\eqref{def-eta2}. Combining this with the coefficient estimates above proves \eqref{Desc:Phi-comp}.


In view of Proposition \ref{Prop:hat_x,y} and Proposition \ref{prop:Desc:Phi-comp}, and \textbf{Lines~6--12} in Algorithm \ref{FSIPL-Line search-comp}, it is easy to see that the Proposed FSIPL method is well defined, i.e., for each $k = 0,1,\dots$, the inner loop of Algorithm \ref{FSIPL-Line search-comp} terminates finitely. We present the uniform termination bound in the next proposition.

\begin{proposition}\label{prop:eta_min,tau_min-comp}
Assume that $\alpha$ satisfies \eqref{def-alpha}. For each $k = 0,1,\dots$, 
the inner loop of Algorithm \ref{FSIPL-Line search-comp} terminates at most 
\begin{equation}
    J:=\max\left\{0,\left\lceil \log_{\frac{\gamma}{2}}\frac{\hat{\eta}_2(\overline{t},\overline{\Delta})}{\overline{\eta}}\right\rceil,\left\lceil \log_{\gamma} \frac{\hat{\tau}_2}{\overline{\tau}}\right\rceil,\left\lceil \log_{\frac{1}{2}}\frac{\overline \tau C_1\kappa}{2\overline \eta D(\overline{t},\overline{\Delta})}\right\rceil \right\}
\end{equation}
inner iterations with $D(\overline{t},\overline{\Delta})$ defined in \eqref{def-eta1-D}, $\hat{\tau}_2$ defined in \eqref{ub:eta,tau-comp-2} and  $\hat{\eta}_2(\overline{t},\overline{\Delta})$ defined \eqref{def-eta2}.
 \end{proposition}

\proof{Proof.}
The backtracking starts from $(\eta_k,\tau_k)=(\overline{\eta},\overline{\tau})$ at Line~5. At every rejection Line~11 sets $\eta_k\leftarrow\frac{\gamma}{2}\eta_k$ and $\tau_k\leftarrow\gamma\tau_k$, so after $j$ rejections
\[
\eta_k=\overline{\eta}(\gamma/2)^j,\qquad \tau_k=\overline{\tau}\gamma^j.
\]
Since $t_k\leq\overline{t}$, $\Delta_k\leq\overline{\Delta}$, and $d(t_k,\Delta_k)$ is nondecreasing in both arguments, we have $D(t_k,\Delta_k)\leq D(\overline{t},\overline{\Delta})$ and $\hat{\eta}_2(t_k,\Delta_k)\geq\hat{\eta}_2(\overline{t},\overline{\Delta})>0$, while $\hat{\tau}_2>0$. Hence $\eta_k\leq\hat{\eta}_2(t_k,\Delta_k)$ holds once
$j\geq J_\eta:=\left\lceil\log_{\gamma/2}\frac{\hat{\eta}_2(\overline{t},\overline{\Delta})}{\overline{\eta}}\right\rceil$,
and $\tau_k\leq\hat{\tau}_2$ (hence $\tau_k\leq\hat{\tau}$) holds once
$j\geq J_\tau:=\left\lceil\log_{\gamma}\frac{\hat{\tau}_2}{\overline{\tau}}\right\rceil$.
It remains to satisfy the coupled bound. Substituting the two recursions into $\eta_k\leq\tau_k C_1\kappa/[2D(t_k,\Delta_k)]$ and using $D(t_k,\Delta_k)\leq D(\overline{t},\overline{\Delta})$, we obtain
\[
\overline{\eta}(\gamma/2)^j\leq\overline{\tau}\gamma^j\frac{C_1\kappa}{2D(\overline{t},\overline{\Delta})}
\iff j\geq J_c:=\left\lceil\log_{1/2}\frac{\overline{\tau}C_1\kappa}{2\overline{\eta}D(\overline{t},\overline{\Delta})}\right\rceil.
\]
Thus, after at most $J=\max\{0,J_\eta,J_\tau,J_c\}$ rejections, Propositions~\ref{Prop:hat_x,y} and~\ref{prop:Desc:Phi-comp} guarantee that both the descent test \eqref{Desc:Phi-comp-L} and $\|h(\hat{x}^{k+1})\|\leq\kappa$ hold, so Line~9 accepts. In particular, the accepted stepsizes satisfy $\eta_k\geq\overline{\eta}(\gamma/2)^J$ and $\tau_k\geq\overline{\tau}\gamma^J$.
\end{proof}

Next, we establish the outer iteration complexity of FSIPL method. Let $\varepsilon > 0$ be a given target accuracy, we provide a bound on $T(\varepsilon)$, which denotes the first outer iteration index to achieve an $\varepsilon$-stationary point (Definition \ref{def:optimization-stationary} (ii)), i.e., let \begin{equation}\label{def:complexity}
     T(\varepsilon):=\min\{k : x^k \text{ is an $\varepsilon$-stationary point of problem \eqref{prob:primal} } \}.
 \end{equation}
  
\begin{theorem}\label{the:complexity-comp}
Assume that $\alpha$ satisfies \eqref{def-alpha}.
Let $\{x^k:k\in\mathbb{N}\}$ be the solution sequence generated by Algorithm \ref{FSIPL-Line search-comp}.   Then, there exists $U>0,V>0$ such that the following inequality holds:
\begin{equation}\label{ineq:complexity-comp}
    T(\varepsilon)\leq \max\{\frac{UV^2}{\varepsilon^2},\frac{UV}{\varepsilon}\}.
\end{equation}
\end{theorem}
\begin{proof}
Let\begin{equation*}\label{r_k-comp}
      r_k:=\frac{\sigma}{2}(\eta_k^2\|d^k\|^2+\tau_k^2\|h(y^k)\|)+\frac{\eta_k}{2t_k}\|v^k\|^2.
  \end{equation*} 
Invoking the search step in Algorithm \ref{FSIPL-Line search-comp}, the following inequality holds:
\begin{equation}\label{ineq-Phi-desc}
    \Phi_{\alpha}(x^{k+1})\leq \Phi_{\alpha}(x^k)-r_k+\rho_k.
\end{equation}
Dropping the nonpositive term $-r_k$ and using the summability of  $\{\rho_k\}$, we obtain
\begin{equation*}\label{def-Phi-upper}
    \Phi_\alpha(x^K)
\leq
\overline{\Phi}_\alpha:=\Phi_\alpha(x^0)
+\widetilde{\rho}
\quad\text{for all }K\geq0,\qquad \widetilde{\rho}:=\sum_{k=0}^{\infty}\rho_k<\infty
\end{equation*}
Since $F(x^K)\leq\Phi_\alpha(x^K)$ and 
$x^K\in\mathcal{M}_{\theta}$, Assumption \ref{assp:N}(vi) shows that
\[
\{x^K:K\in\mathbb N\}
\subseteq
\{x:F(x)\leq\overline{\Phi}_\alpha\}\cap\mathcal{M}_{\theta}
\]
is bounded. In addition, $\Phi_\alpha$ is bounded from below on
$\mathcal{M}_{\theta}$. Indeed, otherwise there would exist
$w^j\in\mathcal{M}_{\theta}$ with $\Phi_\alpha(w^j)\to-\infty$; then
$F(w^j)\leq\Phi_\alpha(w^j)\leq\Phi_\alpha(w^1)$ for all sufficiently large $j$, so Assumption \ref{assp:N}(vi) would make $\{w^j\}$ bounded, contradicting the continuity of $\Phi_\alpha$ after passage to a convergent subsequence. Let
$\underline{\Phi_\alpha}\in\mathbb R$ denote such a lower bound.
Summing up inequality \eqref{ineq-Phi-desc} for $k=0,\dots,K-1$ and reformulating properly, we obtain
\begin{equation*}
   \sum_{k=0}^{K-1}r_k
    \leq \Phi_{\alpha}(x^0)-\Phi_{\alpha}(x^{K})+\sum_{k=0}^{K-1}\rho_k.
\end{equation*}
Let $\eta_*:=\overline{\eta}(\frac{\gamma}{2})^J$ and $\tau_*:=\overline{\tau}\gamma^J$ with $J$ defined in Proposition \ref{prop:eta_min,tau_min-comp}.
It follows that
\begin{equation}\label{ineq:sum}
    \begin{aligned}
   & \sum_{k=0}^{K-1}(\|d^k\|^2+\|v^k\|^2+\|h(y^k)\|)\\
    \leq &\hat{a}\left(\Phi_{\alpha}(x^0)-\Phi_{\alpha}(x^{K})+\sum_{k=0}^{K-1}\rho_k\right),
\end{aligned}
\end{equation}
where
\[ \hat{a}= \max\left\{\frac{2}{\sigma\eta_*^2},\frac{2}{\sigma \tau_*^2},\frac{2\overline{t}}{\eta_*}\right\}. \]
In view of Assumption \ref{assp:N}(i) and \eqref{ineq-res-lem},
\[\|h(x^k)\|\leq \|h(y^k)\|+\eta_k\Delta_k+\frac{L_h}{2}\eta_k^2\|d^k\|^2,\] and then
\begin{equation}\label{ineq:d^2,h}
    \begin{aligned}
    &\sum_{k=0}^{K-1}\max\{\|d^k\|^2,\|v^k\|^2,\|h(x^k)\|\}\\\leq & \sum_{k=0}^{K-1}\max\{\|d^k\|^2,\|v^k\|^2,\|h(y^k)\|+\eta_k\Delta_k+\frac{L_h}{2}\eta_k^2\|d^k\|^2\}\\
    \leq& (\frac{L_h}{2}\overline{\eta}^2+1)\sum_{k=0}^{K-1}(\|d^k\|^2+\|v^k\|^2+\|h(y^k)\|)+\overline{\eta}\sum_{k=0}^{K-1}\Delta_k.
\end{aligned}
\end{equation}
Combining \eqref{ineq:d^2,h} with  \eqref{ineq:sum}, we obtain 
\begin{equation}\label{prf-max d, h}
    \begin{aligned}
    &\sum_{k=0}^{K-1}\max\{\|d^k\|^2,\|v^k\|^2,\|h(x^k)\|\}\\
    \leq & (\frac{L_h}{2}\overline{\eta}^2+1)\hat{a}\left(\Phi_{\alpha}(x^0)-\Phi_{\alpha}(x^{K})+\sum_{k=0}^{K-1}\rho_k\right)+\overline{\eta}\sum_{k=0}^{K-1}\Delta_k,
\end{aligned}
\end{equation}

From \eqref{eq-optimal-d,v}, we have
\begin{equation*}
    \begin{aligned}
    &\mathrm{dist}(0,\nabla f(x^k)+\nabla\mathcal{A}(x^k)^{\top}\partial g\big(\mathcal{A}(x^k)+v^k\big)+ \nabla h(x^k)^{\top}\lambda^k)\\
    \leq &\|\nabla f(x^k)+\nabla\mathcal{A}(x^k)^{\top}(\mu^k-\frac{v^k}{t_k})+\nabla h(x^k)^{\top}\lambda^k\|
    \leq \frac{\|d^k\|+\ell_{\mathcal{A}}\|v^k\|}{t_k},
    \end{aligned}
\end{equation*}
and hence, by the definition of $Res$ in Definition~\ref{def:optimization-stationary},
\begin{equation}\label{ineq:Res}
    Res(x^k;v^k,\lambda^k)\leq V\max\{\|d^k\|,\|v^k\|,\|h(x^k)\|\}\quad\text{for all }k,\qquad V:=V_1:=1+\frac{1+\ell_{\mathcal A}}{\underline t}.
\end{equation}

 Let 
\begin{equation}\label{eq:U}
U:=\hat{a}\Big(\frac{L_h}{2}\overline{\eta}^{2}+1\Big)
\big(\Phi_{\alpha}(x_0)-\underline{\Phi_{\alpha}}+\widetilde{\rho}\big)
+\overline{\eta}c_{\Delta}\widetilde{\rho}.
\end{equation}
Then, from \eqref{prf-max d, h} and  \eqref{ineq:Res}, we deduce that
\begin{equation}
    \min_{0\leq k\leq K-1} Res(x^k;v^k,\lambda^k)\leq V\max\{\frac{U^{1/2}}{K^{1/2}}, \frac{U}{K}\},
\end{equation}
Hence, an $\varepsilon$-stationary point is found within 
\[
K:=\left\lfloor\max\Big\{\frac{UV^{2}}{\varepsilon^{2}},\frac{UV}{\varepsilon}\Big\}\right\rfloor+1
\]
outer iterations, which implies that 
\[
T(\varepsilon)\le K-1
\le\max\Big\{\frac{UV^{2}}{\varepsilon^{2}},\frac{UV}{\varepsilon}\Big\}.
\]
\end{proof}

This result shows that the proposed FSIPL method can find an $\varepsilon$-stationary point of problem \eqref{prob:primal} within $O(1 / \varepsilon^2)$ outer iterations. Although the order $O(1 / \varepsilon^2)$is the same as that of feasible manifold proximal gradient-type methods, the present result is obtained for an infeasible and inexact scheme, without requiring exact solution of the proximal subproblem.


Finally, we prove the subsequential convergence of the proposed method.
\begin{theorem}\label{the:convger-comp}
Assume that $\alpha$ satisfies \eqref{def-alpha}.   Let $\{x^k:k\in\mathbb{N}\}$ be the solution sequence generated by Algorithm \ref{FSIPL-Line search-comp}.  Then,  any cluster point of $\{x^k:k\in\mathbb{N}\}$ is a stationary point of problem \eqref{prob:primal}.
\end{theorem}
\begin{proof}
 Let $K\to\infty$ in \eqref{prf-max d, h}, we obtain
\[\sum_{k=0}^{\infty}\|d^k\|^2<\infty,~~\sum_{k=0}^{\infty}\|v^k\|^2<\infty,~~\sum_{k=0}^{\infty}\|h(x^k)\|<\infty,\]
which implies that $\lim_{k\to\infty} d^k=0$, $\lim_{k\to\infty} v^k=0$ and $\lim_{k\to\infty} \|h(x^k)\|=0$. Since $x^0\in\mathcal M$ and the acceptance test in Algorithm~\ref{FSIPL-Line search-comp} enforces $x^k\in\mathcal M_\kappa\subset\mathcal M_\theta$ for every iterate, Lemma~\ref{lemma:d^k-bound-comp} implies that the multipliers $(\lambda^k,\mu^k)$ are uniformly bounded by $\lambda_{\max}$ and $\mu_{\max}$ defined in \eqref{def:uniform-const}. The uniform objective upper bound $\Phi_\alpha(x^K)\leq\overline{\Phi}_\alpha$ and Assumption~\ref{assp:N}(vi) then imply that $\{x^k\}$ is bounded, so the cluster point set is nonempty. Let $x^*$ be any cluster point and take a subsequence $x^k\to x^*$; by the boundedness of $(\lambda^k,\mu^k)$, pass to a further subsequence with $(\lambda^k,\mu^k)\to(\lambda^*,\mu^*)$. Passing to the limit in \eqref{eq-optimal-d,v} along this subsequence, and using $d^k\to0$, $v^k\to0$, the continuity of $\nabla f$, $\nabla h$ and $\nabla\mathcal A$, and the closedness of $\partial g$, we obtain
\[
0\in \nabla f(x^*)+\nabla\mathcal A(x^*)^\top\partial g(\mathcal A(x^*))
+\nabla h(x^*)^\top\lambda^* .
\]
Together with $h(x^*)=0$, this shows that $x^*$ is a stationary point of problem~\eqref{prob:primal}.
\end{proof}

\section{Sequential convergence}\label{sec:converge}  
In this section, we establish the convergence of the full sequence generated by  FSIPL (Algorithm \ref{FSIPL-Line search-comp}) under the Kurdyka--\L{}ojasiewicz (KL) property assumption. To this end, we first review the notions of KL property (\citet{attouch2010proximal}) and uniformized KL property (\citet{bolte2014proximal}). 
\begin{definition}\label{lemma:KL}
(KL property). A proper function $ \psi : \mathbb{R}^n \to (-\infty, +\infty] $ is said to satisfy the KL property at $ x \in \text{dom}(\partial \psi) $ if there exist $ \varepsilon \in (0, +\infty] $, $ \delta > 0 $, and a continuous concave function $ \phi : [0, \varepsilon) \to \mathbb{R}_+ := [0, +\infty) $, such that:  

\begin{enumerate}
    \item[(i)] $ \phi(0) = 0 $;  
    \item[(ii)] $ \phi $ is continuously differentiable on $ (0, \varepsilon) $ with $ \phi' > 0 $;  
    \item[(iii)] For any $ z \in \mathcal{B}(x, \delta) \cap \{ z \in \mathbb{R}^n : \psi(x) < \psi(z) < \psi(x) + \varepsilon \} $, there holds  
    \[
    \phi' \left( \psi(z) - \psi(x) \right) \mathrm{dist}(0, \partial \psi(z)) \geq 1.
    \]
\end{enumerate}

\end{definition}  
    
\begin{lemma}\label{lemma:uniform^kL}
(Uniformized KL property). Let $ \Upsilon \subseteq \mathbb{R}^n $ be a compact set, and let the proper function $ \psi : \mathbb{R}^n \to (-\infty, +\infty] $ be constant on $ \Upsilon $. If $ \psi $ satisfies the KL property at each point of $ \Upsilon $, then there exist $ \varepsilon, \delta > 0 $ and a continuous concave function $ \phi : [0, \varepsilon) \to [0, +\infty) $ satisfying Definition \ref{lemma:KL} (i) and (ii) such that  
\[
\phi' \left( \psi(z) - \psi(x) \right) \, \text{dist}(0_n, \partial \psi(z)) \geq 1 
\]  
holds for any $ x \in \Upsilon $ and $ z \in \mathcal{B}(x, \delta) \cap \left\{ z \in \mathbb{R}^n : \psi(x) < \psi(z) < \psi(x) + \varepsilon \right\} $.
\end{lemma}
A proper function $ \psi : \mathbb{R}^n \to (-\infty, +\infty] $ is called a KL function if it satisfies the KL property at any point in $ \text{dom}(\partial \psi)$. Recall that the important class of semialgebraic functions is known to satisfy the nonsmooth KL 
property (\citet{bolte2014proximal}). The KL-based framework for proving sequential convergence of descent algorithms for nonconvex optimization was first established in \citet{attouch2010proximal,attouch2013convergence}, and has been widely used and extended in the literature. Below we present an extension of the original frameworks in  Proposition \ref{prop:kl-converge} which best matches the setting of this paper. For readability, we defer the proof to Appendix \ref{append-proof-KL}.
\begin{proposition}\label{prop:kl-converge}
    Let $ \psi: \mathbb{R}^{q_1}\times \mathbb{R}^{q_2} \to (-\infty, \infty] $ be a proper lower semicontinuous function. Consider a bounded sequence $\{(u^{k},\upsilon^k)\in\mathbb{R}^{q_1}\times \mathbb{R}^{q_2}  :  k\in\mathbb{N}\}$ and a nonnegative sequence $\{\xi_k\in\mathbb{R}_{+} :  k\in\mathbb{N}\}$ satisfying the following three conditions.
    \begin{enumerate}
        \item[(a)] (Descent Property)
        There exist $c_0>0$, $c_1>0$ and $k_1\in\mathbb{N}$ such that, for all $k\geq k_1$,  
        \begin{equation*}
        \|u^{k+1}-u^{k}\|\leq c_0\xi_{k+1} ~\text{ and }~
            c_1\xi_{k+1}^2\leq \psi(u^k,\upsilon^k)-\psi(u^{k+1},\upsilon^{k+1}).
        \end{equation*}
        
        \item [(b)] (Relative Error)
        There exist $c_2>0$ and $k_2\in\mathbb{N}$ such that, for all $k\geq k_2$,
        \begin{equation*}
            \mathrm{dist}(0,\partial \psi(u^{k+1},\upsilon^{k+1}))\leq c_2\xi_{k+1}.
         \end{equation*}
       \item [(c)] (Continuity)
              The limit $\psi_{\infty}:=\lim_{k\to\infty} \psi(u^{k},\upsilon^k)$ exists and $\psi\equiv \psi_{\infty}$ on $\Lambda$, where $\Lambda$ is the set of cluster points of the sequence $\{(u^{k},\upsilon^k)  :  k\in\mathbb{N}\}$. 
    \end{enumerate}
  If $\psi$ satisfies the KL property at each point within $\Lambda$, then  $\sum_{k=0}^{\infty} \xi_{k}<\infty$ and it follows that $\sum_{k=0}^{\infty} \|u^{k+1}-u^k\|<\infty$, $0\in\partial\psi(u^*,\upsilon^*)$ where $u^*=\lim_{k\to\infty} u^k$ and $\upsilon^*$ is any cluster point of $\{\upsilon^k\}$. 
\end{proposition}

We now apply Proposition~\ref{prop:kl-converge} to FSIPL. The main issue is not
the abstract KL argument itself, but the construction of a
function whose descent and subgradient estimates are both
compatible with the infeasible and inexact iterates.
The approximate optimality system \eqref{eq-optimal-d,v} provides a
subgradient of $g$ at
$\mathcal A(x^k)+v^k$, whereas the merit function
$\Phi_\alpha$ contains $g(\mathcal A(x^k))$.
Closeness of these two arguments does not in general
provide a corresponding estimate between the required
subgradients. Moreover, the multiplier term
$\nabla h(x^k)^\top\lambda^k$ need not coincide with a
subgradient of the penalty $\alpha\|h(x^k)\|$.
Consequently, the available optimality system does not
directly yield the relative-error condition for
$\Phi_\alpha$.

We therefore construct the KL argument around the
quantities actually produced by the algorithm. Lemma~\ref{lem:Desc:Q}
first identifies a descent quantity involving the shifted
composite value, the feasibility residual, and correction
and linearization terms. We then represent this quantity,
together with suitable memory terms, as the value of an
auxiliary function on an extended sequence. The additional
parameter conditions in this section ensure acceptance
of full tangential steps and permit the inexactness terms
to be absorbed into the descent estimate. The resulting
full-sequence conclusion is established for this parameter
regime, rather than for arbitrary backtracking parameters.
\begin{lemma}\label{lem:Desc:Q}
  Let $\{(x^k,y^k,d^k,v^k,\lambda^k,\mu^k):k\in\mathbb{N}\}$ be generated by Algorithm \ref{FSIPL-Line search-comp} with $\alpha$ satisfying \eqref{def-alpha}, $\overline{\eta}=1$, $\overline{\tau}\leq \hat{\tau}_2$, $c_{\Delta}\leq \frac{1}{C_e}$, and $t_k,\Delta_k$ satisfying  $D(t_k,\Delta_k)\leq\frac{\overline\tau C_1\kappa}{2}$ and $\hat{\eta}_2(t_k,\Delta_k)\geq 1$. 
 Then, for every $k\geq 1$, the following inequality holds:
  \begin{equation}\label{Desc:Q}
      \begin{aligned}
         &Q_{k+1}-Q_{k}
         \leq  -\frac{\sigma}{2}(\|d^{k}\|^2+\overline{\tau}^2\|h(y^k)\|)-\frac{1}{2t_k}\|v^k\|^2+(\alpha+\|\lambda^{k}\|+\|\mu^{k}\|)\Delta_{k}+\ell_g\Delta_{k-1},
    \end{aligned}
  \end{equation}
  where 
  \begin{align}\label{def-Q}
      Q_{k}=f(x^{k})+g\big(\mathcal{A}(x^{k-1})+v^{k-1}\big)+\frac{\ell_gL_{\mathcal{A}}}{2}\|d^{k-1}\|^2+\ell_g\ell_{\mathcal{A}}\|x^{k}-y^{k-1}\|+\alpha\|h(x^{k})\|.
  \end{align}
\end{lemma}
\begin{proof}

By the choice of $t_k$ and $\Delta_k$, \eqref{x,y_range-comp} in Proposition \ref{Prop:hat_x,y}, and Proposition \ref{prop:Desc:Phi-comp}, it holds that for all $k\geq 1$, $\eta_k=\overline{\eta}=1$, $\tau_k=\overline{\tau}$ and  $(x^k,y^k)\in\mathcal{M}_{\theta}\times\mathcal{M}_{\theta}$. For every $k\geq1$, the convexity of $g$ together with \eqref{def:d^k-comp} and \eqref{def:v^k-comp} yields that
\begin{equation}\label{ineq:prop4-prox-comp}
    \begin{aligned}
&g(\mathcal{A}(x^{k})+v^{k})+\langle \nabla f(x^{k}),d^{k}\rangle+\frac{1}{2t_k}\|d^{k}\|^2+\frac{1}{2t_k}\|v^{k}\|^2\\
       \leq& g\big(\mathcal{A}(x^{k})\big)-\langle \lambda^{k}, \nabla h(x^{k})d^{k}\rangle-\langle\mu^{k},\nabla\mathcal{A}(x^{k})d^{k}-v^{k}\rangle\\
       \leq & g\big(\mathcal{A}(x^{k})\big)+(\|\lambda^k\|+\|\mu^k\|)\Delta_k,
    \end{aligned}
\end{equation}
where the last inequality comes from \eqref{ineq-res-lem}.
We derive from Assumption \ref{assp:N} (ii) that \begin{equation}
        f(y^{k})\leq f(x^{k})+\langle \nabla f(x^{k}), d^k\rangle +\frac{L_f}{2}\|d^k\|^2,\label{ineq:prop4.4-f1}
    \end{equation}
    and by the Lipschitz continuity of $f$, 
    \begin{equation}\label{ineq:prop4.4-f2}
        f(x^{k+1})\leq f(y^k)+\ell_f\|x^{k+1}-y^k\|.
    \end{equation}
    Note that by the Lipschitz continuity of $\mathcal{A}$ and $\nabla \mathcal{A}$, it holds that 
    \begin{equation}\label{ineq:Ax-}
        \begin{aligned}
    &\|\mathcal{A}(x^k)-(\mathcal{A}(x^{k-1})+v^{k-1})\|\\
    \leq&\|\mathcal{A}(x^k)-\mathcal{A}(y^{k-1})\|+\|\mathcal{A}(y^{k-1})-(\mathcal{A}(x^{k-1})+\nabla\mathcal{A}(x^{k-1})d^{k-1})\|\\
    &+\|\nabla\mathcal{A}(x^{k-1})d^{k-1}-v^{k-1}\|\\
    \leq& \ell_{\mathcal{A}}\|x^k-y^{k-1}\|+\frac{L_{\mathcal{A}}}{2}\|d^{k-1}\|^2+\Delta_{k-1},
\end{aligned}
    \end{equation}
then it follows from the Lipschitz continuity of $g$ that 
\begin{equation}\label{ineq:prop4.4-g}
    \begin{aligned}
       & g\big(\mathcal{A}(x^k)\big)-g\big(\mathcal{A}(x^{k-1})+v^{k-1}\big)\\\leq &\ell_g\|\mathcal{A}(x^k)-(\mathcal{A}(x^{k-1})+v^{k-1})\|\\
        \leq&\ell_g\ell_{\mathcal{A}}\|x^k-y^{k-1}\|+\frac{\ell_gL_{\mathcal{A}}}{2}\|d^{k-1}\|^2+\ell_g\Delta_{k-1}.
    \end{aligned}
\end{equation}
    Summing \eqref{ineq:prop4-prox-comp}, \eqref{ineq:prop4.4-f1}, \eqref{ineq:prop4.4-f2} and \eqref{ineq:prop4.4-g}, we obtain
    \begin{equation}\label{ineq:prop4-1}
        \begin{aligned}    &f(x^{k+1})+g\big(\mathcal{A}(x^k)+v^k\big)+\frac{\ell_gL_{\mathcal{A}}}{2}\|d^k\|^2+\ell_g\ell_{\mathcal{A}}\|x^{k+1}-y^k\|\\
        \leq &f(x^k)+g\big(\mathcal{A}(x^{k-1})+v^{k-1}\big)+\frac{\ell_gL_{\mathcal{A}}}{2}\|d^{k-1}\|^2+\ell_g\ell_{\mathcal{A}}\|x^k-y^{k-1}\|+\ell_g\Delta_{k-1}\\
        &-\frac{1}{2}(\frac{1}{t_k}-L_f-\ell_gL_{\mathcal{A}})\|d^k\|^2-\frac{1}{2t_k}\|v^k\|^2+(\|\lambda^k\|+\|\mu^k\|)\Delta_{k}+(\ell_f+\ell_g\ell_{\mathcal{A}})\|x^{k+1}-y^k\|.
    \end{aligned}
    \end{equation}
    Combining the item \eqref{ineq:Desc:h(x)} in Proposition \ref{Prop:hat_x,y} and \eqref{ineq:prop4-1}, following the similar argument as the proof of Proposition \ref{prop:Desc:Phi-comp}, we obtain \eqref{Desc:Q}.
\end{proof}

Lemma~\ref{lem:Desc:Q} shows that the available descent estimate is not expressed solely in terms of the merit function $\Phi_\alpha$. Instead, it involves the shifted composite value $g(\mathcal A(x^k)+v^k)$, the feasibility residual, the linear approximation errors, the inexactness error, and memory terms. This estimate suggests what quantities should be incorporated into the auxiliary function. On the other hand, the relative-error condition in Proposition~\ref{prop:kl-converge} determines
how these quantities should be embedded into an extended-variable
function. Motivated by these two requirements, we introduce the following auxiliary function $P:\mathbb{R}^n\times\mathbb{R}^p\times\mathbb{R}^m\times\mathbb{R}^p\times\mathbb{R}\to(-\infty,+\infty]$ defined at $(x,z,\upsilon,s,\omega)$ as  
\begin{equation}\label{def:P}  P(x,z,\upsilon,s,\omega):=f(x)+g(z)+\delta_{\mathcal{C}}(x,\upsilon)+\delta_{\mathcal{D}}(x,z,s)+\omega^2,
\end{equation} where
\begin{equation}\label{def:C}
    \mathcal{C}:=\{(x,\upsilon)\in\mathbb{R}^n\times\mathbb{R}^m :  h_i(x)=\upsilon_i^3,~i=1,\dots,m\},
\end{equation}
and \begin{equation}\label{def:D}
    \mathcal{D}:=\{(x,z,s)\in\mathbb{R}^n\times\mathbb{R}^p\times \mathbb{R}^p :  a_j(x)-z_j=s_j^3,~j=1,\dots,p\}
\end{equation}
with $\mathcal{A}(x)=[a_1(x),\dots,a_p(x)]^{\top}$.
Obviously, $P$ is a proper lower semicontinuous function. Moreover, for the points considered below, the normal cone of $\mathcal{C}$  at $(x,\upsilon)$ and the normal cone of $\mathcal{D}$ at $(x,z,s)$, are respectively given by
\[N_{\mathcal{C}}(x,\upsilon)=\left\{\begin{bmatrix}
    \nabla h(x)^{\top}\\-3\mathrm{diag}(\upsilon_i^2)
\end{bmatrix}\lambda :  \lambda\in\mathbb{R}^m\right\},~~N_{\mathcal{D}}(x,z,s)=\left\{\begin{bmatrix}
    \nabla \mathcal{A}(x)^{\top}\\
    -I_p\\-3\mathrm{diag}(s_j^2)
\end{bmatrix}\mu :  \mu\in\mathbb{R}^p\right\}\]
where \[\mathrm{diag}(\upsilon_i^2)=\begin{bmatrix}
    \upsilon_1^2&&\\
    &\ddots\\
    &&\upsilon_{m}^2
\end{bmatrix},~~\mathrm{diag}(s_j^2)=\begin{bmatrix}
    s_1^2&&\\
    &\ddots\\
    &&s_{p}^2
\end{bmatrix}.\]
Each auxiliary variable in $P$ serves a specific purpose.
The variable $z$ retains the shifted composite argument
at which the proximal optimality system supplies a
subgradient of $g$. The graph constraints
$h_i(x)=\upsilon_i^3$ and
$\mathcal A_j(x)-z_j=s_j^3$ reproduce the multiplier
terms $\nabla h(x)^\top\lambda$ and
$\nabla\mathcal A(x)^\top\mu$ in their normal vectors.
The cubic representation accommodates signed residuals
and makes the additional normal components proportional
to $\upsilon_i^2\lambda_i$ and $s_j^2\mu_j$.
These components involve residual powers of $2/3$ and,
using the uniform multiplier bounds and the tolerance
conditions, can be controlled by the square-root residual
quantities entering the relative-error estimate.

The scalar $\omega$ represents the remaining nonnegative
correction and memory terms in the descent quantity.
With the choice in \eqref{def:nu-omega-xi}, the value of $P$ equals
$Q_k+\sigma S_{k-1}/4+\sigma S_{k-2}/8$ along the
extended sequence. Proposition~\ref{prop:kl-cond-1} then verifies the
relative-error estimate, while Proposition~\ref{prop:kl-cond-2} establishes
sufficient descent and continuity on the cluster set.
The construction thus connects the available primal--dual
optimality system to the quantities controlled by the
descent analysis.

Suppose that the sequence $\{(x^k,y^k,d^k,v^k,\lambda^k,\mu^k):k\in\mathbb{N}\}$ is generated by Algorithm \ref{FSIPL-Line search-comp}. For $k\geq 1$, let \mbox{$S_{k-1}:=\max\left\{\overline{\tau}^2\|h(y^{k-1})\|,\|v^{k-1}\|^2,\|d^{k-1}\|^2\right\}$}, and for $k\geq 2$, let  
 \begin{equation}\label{def:nu-omega-xi}
        \begin{cases}
            {\begin{aligned}
    z^k:=&\mathcal{A}(x^{k-1})+v^{k-1},\\     
\upsilon^k_i:=&[h_i(x^k)]^{1/3},\\
s_j^k:=&[a_j(x^{k})-z_j^{k}]^{1/3},\\
\omega_k:=&\sqrt{\alpha\|h(x^k)\|+\frac{\ell_gL_\mathcal{A}}{2}\|d^{k-1}\|^2+\frac{\sigma}{4}S_{k-1}+\frac{\sigma}{8}S_{k-2}+\ell_g\ell_{\mathcal{A}}\|x^{k}-y^{k-1}\|}.
\end{aligned}}
        \end{cases} 
\end{equation} 
The following propositions demonstrate that the sequence generated by FSIPL satisfies the descent, relative error, and continuity conditions needed for Proposition~\ref{prop:kl-converge}, under the conditions of Lemma \ref{lem:Desc:Q} with 
\begin{equation}\label{settings-Delta}
\begin{aligned}
 \underline{t}\leq  t_k&\leq \overline{t}\leq \min\{\frac{\sqrt{\overline{b}}}{2\sqrt{L_h}(2\ell_f+\ell_g)},~\frac{1}{L_{\mathrm{mod}}}\},\\
        \Delta_{k}&\leq \min\left\{\Delta_{k-1},~\overline{\Delta},~\frac{\min\{\frac{\sigma}{8}S_{k-1},~\rho_k\}}{C_e} \right\},
         \end{aligned}
     \end{equation}
     where $\overline{b}:=\min\{\theta-\kappa,\frac{\overline{\tau}C_1\kappa}{2}\}$ and $\overline{\Delta}=\min\Big\{\frac{\overline{b}}{2},\frac{\sqrt{\overline{b} C_1}}{2\sqrt{L_h}(\ell_{\mathcal{A}}+1+\sqrt{C_1})}\Big\}$, $L_{\mathrm{mod}}$ and $C_e$ are defined in \eqref{ub:eta,tau-comp-2} and \eqref{def:uniform-const} respectively. We choose $\underline t>0$ so that the interval for $t_k$ is nonempty, and set initial pair with $t_0\in[\underline t,\overline t]$ and $0\le\Delta_0\le\min\{\overline\Delta,c_\Delta\rho_0\}$. Then, the full-step conditions already hold at $k=0$ and the estimates involving $y^0,d^0,v^0$ are initialized consistently. By direct calculation, the setting \eqref{settings-Delta} satisfies $D(t_k,\Delta_k)\leq\frac{\overline\tau C_1\kappa}{2}$ and $\hat{\eta}_2(t_k,\Delta_k)\geq 1$ with $c_{\Delta}=\frac{1}{C_e}$ in Lemma \ref{lem:Desc:Q}: the first bound on $t_k$ together with the choice of $\overline{\Delta}$ gives $\Delta_k+\frac{L_h}{2}d(t_k,\Delta_k)^2\leq \overline{b}$ and hence $\hat{\eta}_1(t_k,\Delta_k)\geq1$, the bound $t_k\leq 1/L_{\mathrm{mod}}$ gives $t_kL_{\mathrm{mod}}\leq1$, and $c_{\Delta}=1/C_e$ gives $c_{\Delta}C_e=1$.


 \begin{proposition}\label{prop:kl-cond-1}
     Let $\{(x^k,y^k,d^k,v^k,\lambda^k,\mu^k):k\in\mathbb{N}\}$ be generated by Algorithm \ref{FSIPL-Line search-comp} under the condition of Lemma \ref{lem:Desc:Q} with $t_k$ and  $\Delta_k$ satisfying \eqref{settings-Delta}.
     Let $P$ be defined by \eqref{def:P} and $(z^k,\upsilon^k,s^k,\omega_k)\in\mathbb{R}^{p}\times\mathbb{R}^m\times\mathbb{R}^p\times\mathbb{R}$ be defined as \eqref{def:nu-omega-xi}. Then, there exist $c_0>0$ and $c_2>0$ such that for all sufficiently large $k$, $\|x^{k+1}-x^k\|\leq c_0\xi_{k+1}$ and 
     \begin{equation}\label{ineq-partial}
    \mathrm{dist}\bigl(0,\partial P(x^{k+1},z^{k+1},\upsilon^{k+1},s^{k+1},\omega_{k+1})\bigr)\leq c_2\xi_{k+1},
     \end{equation}  
     where $\xi_{k+1}=\sqrt{\|d^{k}\|^2+\|v^k\|^2+\|h(y^k)\|+S_{k-1}}$.
 \end{proposition}
\begin{proof}
Since the conditions of Lemma \ref{lem:Desc:Q} satisfy the conditions of Proposition \ref{prop:Desc:Phi-comp}, it comes from the proof of Proposition \ref{prop:Desc:Phi-comp}  that
\begin{equation}\label{prf:the:x-y}
    \|x^{k+1}-y^k\|\leq  \hat{\tau}\sqrt{C_2}\|h(y^k)\|.
\end{equation}
Then, by the update of $y^{k}$ as \eqref{step2}, we obtain
\begin{equation}\label{prf-x-x-0}
    \begin{aligned}
   \|x^{k+1}-x^k\|&\leq\|d^k\|+\hat{\tau}\sqrt{C_2}\|h(y^k)\|\\
   &\leq \sqrt{2\|d^k\|^2+2\hat{\tau}^2C_2\|h(y^k)\|^2}.
\end{aligned}
\end{equation}
Since $y^k\in\mathcal{M}_{\theta}$ as shown in Lemma \ref{lem:Desc:Q},  \eqref{prf-x-x-0} yields that
\begin{equation}\label{prf-x-x}
  \|x^{k+1}-x^k\|\leq \sqrt{2\|d^k\|^2+8\theta \hat{\tau}^2C_2\|h(y^k)\|}.
\end{equation}

Combining \eqref{prf-x-x} and the definition of $\xi_{k+1}$, we conclude that  there exists $c_0>0$ such that for all $k\geq 1$,  
\begin{equation}\label{ineq:x-xi}
    \|x^{k+1}-x^k\|\leq c_0\xi_{k+1}.
\end{equation} 

Next, we prove \eqref{ineq-partial}. To justify the subgradient inclusion below, let
\[
\mathcal E:=\{(x,z,\upsilon,s):(x,\upsilon)\in\mathcal C,\ (x,z,s)\in\mathcal D\}.
\]
Then $\delta_{\mathcal E}=\delta_{\mathcal C}+\delta_{\mathcal D}$, so the
auxiliary function in \eqref{def:P} reads
\[
P(x,z,\upsilon,s,\omega)
=f(x)+g(z)+\delta_{\mathcal E}(x,z,\upsilon,s)+\omega^{2}.
\]
The Jacobian of the defining constraints of $\mathcal E$, with variables
ordered as $(x,z,\upsilon,s)$, is
\[
\begin{pmatrix}
\nabla h(x) & 0 & -3\operatorname{diag}(\upsilon_i^{2}) & 0\\
\nabla\mathcal A(x) & -I_p & 0 & -3\operatorname{diag}(s_j^{2})
\end{pmatrix}.
\]
At any point with $x\in\mathcal M_{\theta}$, this matrix has full row rank.
Indeed, if its transpose annihilates
$(\lambda,\mu)\in\mathbb R^{m}\times\mathbb R^{p}$, then the $z$-component
gives $\mu=0$, and the $x$-component subsequently gives
$\nabla h(x)^{\top}\lambda=0$. Assumption~\ref{assp:N}(i) then implies
$\lambda=0$. Thus $\mathcal E$ is locally a continuously differentiable
submanifold at every such point, and its regular normal cone is the range of
the transpose of the above Jacobian. The iterates $x^{k+1}$ lie in $\mathcal M_{\kappa}\subset\mathcal M_\theta$ by Proposition~\ref{Prop:hat_x,y}(i), so the points at which the inclusion \eqref{in:partial} is taken fall inside this region. Then, the convex subgradient inequality for
$g$, the smoothness of $f$ and $\omega^{2}$, and the regular-normal
inequality for $\mathcal E$  yield the following inclusion for any
$\iota^{k+1}\in\partial g(z^{k+1})$:
\begin{equation}\label{in:partial}
    \begin{aligned}\
    \zeta^{k+1}&:=\begin{bmatrix}
   \nabla f(x^{k+1})+\nabla h(x^{k+1})^{\top}\lambda^{k}+\nabla\mathcal{A}(x^{k+1})^{\top}\mu^k\\
   \iota^{k+1}-\mu^k \\
   -3\mathrm{diag}([{\upsilon^{k+1}_i}]^{2})\lambda^{k}\\
    -3\mathrm{diag}([{s^{k+1}_j}]^{2})\mu^{k}\\
   2\omega_{k+1}
\end{bmatrix}\in \partial P(x^{k+1},z^{k+1},\upsilon^{k+1},s^{k+1},\omega_{k+1}).
\end{aligned}
\end{equation}
 The definition of $d^k$ in \eqref{def:d^k-comp}, together with the uniform multiplier bounds \eqref{ineq:lambda^k-bound-comp}--\eqref{ineq:mu^k-bound-comp} in Lemma~\ref{lemma:d^k-bound-comp} and the constants in \eqref{def:uniform-const}, gives
\begin{equation}
    \begin{aligned}\label{ineq:sub-1}
  &\|\nabla f(x^{k+1})+\nabla h(x^{k+1})^{\top}\lambda^{k}+\nabla\mathcal{A}(x^{k+1})^{\top}\mu^k\|\\
  \leq &\frac{1}{\underline{t}}\|d^{k}\|+(L_f+L_h\lambda_{\max}+L_{\mathcal{A}}\mu_{\max})\|x^{k+1}-x^k\|\\
  \leq& (\frac{1}{\underline{t}}+(L_f+L_h\lambda_{\max}+L_{\mathcal{A}}\mu_{\max})c_0)\xi_{k+1},
\end{aligned}
\end{equation}
where the last inequality uses $\|d^k\|\leq\xi_{k+1}$ and \eqref{ineq:x-xi}.
Likewise, the definition of $v^k$ in \eqref{def:v^k-comp} yields an
$\iota^{k+1}\in\partial g(z^{k+1})$ such that
\begin{equation}\label{ineq:sub-2}
  \|\iota^{k+1}-\mu^{k}\|=\frac{\|v^k\|}{t_k}\leq \frac{1}{\underline{t}}\xi_{k+1}.
\end{equation}

Moreover, note that
\begin{equation*}
\begin{aligned}
    \|\mathrm{diag}([\upsilon_i^{k+1}]^2)\|^2=&\sum_{i=1}^m \big(h_i(x^{k+1})\big)^{4/3}\leq \|h(x^{k+1})\|_{\infty}^{1/3}\|h(x^{k+1})\|_1,\\
    \|\mathrm{diag}([s_j^{k+1}]^2)\|^2=&\sum_{j=1}^p\big(a_j(x^{k+1})-z_j^{k+1}\big)^{4/3}\leq \|\mathcal{A}(x^{k+1})-z^{k+1}\|_{\infty}^{1/3}\|\mathcal{A}(x^{k+1})-z^{k+1}\|_1.
\end{aligned}
\end{equation*}
By \eqref{ineq:Ax-} and \eqref{prf:the:x-y},
\[\|\mathcal{A}(x^{k+1})-z^{k+1}\|\leq \ell_{\mathcal{A}}\hat{\tau}\sqrt{C_2}\|h(y^k)\|+\frac{L_{\mathcal{A}}}{2}\|d^k\|^2+\Delta_k.\]
Since $y^k\in\mathcal{M}_{\theta}$ and $\{d^k\}$ is bounded by Lemma~\ref{lemma:d^k-bound-comp}, $\{\mathcal{A}(x^{k+1})-z^{k+1}\}$ is bounded. By equivalence of norms, there exists $a_1>0$ such that
\begingroup
\setlength{\abovedisplayskip}{3pt plus 1pt minus 1pt}
\setlength{\belowdisplayskip}{3pt plus 1pt minus 1pt}
\setlength{\abovedisplayshortskip}{0pt plus 1pt}
\setlength{\belowdisplayshortskip}{3pt plus 1pt minus 1pt}
\begin{align}
    &\|\mathrm{diag}([\upsilon_i^{k+1}]^2)\lambda^k\|\leq \lambda_{\max}\|\mathrm{diag}([\upsilon_i^{k+1}]^2)\|\leq a_1\lambda_{\max}\sqrt{\|h(x^{k+1})\|}\leq a_1\lambda_{\max}\xi_{k+1},\label{ineq:sub-3}\\
    &\|\mathrm{diag}([s_j^{k+1}]^2)\mu^k\|\leq \mu_{\max}\|\mathrm{diag}([s_j^{k+1}]^2)\|\leq a_1\mu_{\max}\sqrt{\|\mathcal{A}(x^{k+1})-z^{k+1}\|}.\label{ineq:sub-4-1}
\end{align}
\endgroup
The last inequality in \eqref{ineq:sub-3} uses
$\|h(x^{k+1})\|\leq\|h(y^k)\|$ (Proposition~\ref{Prop:hat_x,y}(i)). Combining
\eqref{ineq:Ax-} and \eqref{ineq:sub-4-1},
\begin{equation*}
\begin{aligned}
\|\mathrm{diag}([s_j^{k+1}]^2)\mu^k\|&\leq a_1\mu_{\max}\sqrt{\ell_{\mathcal{A}}\|x^{k+1}-y^k\|+\frac{L_{\mathcal{A}}}{2}\|d^k\|^2+\Delta_k}\\
&\leq a_1\mu_{\max}\sqrt{\ell_{\mathcal{A}}\hat{\tau}\sqrt{C_2}\|h(y^k)\|+\frac{L_{\mathcal{A}}}{2}\|d^k\|^2+\Delta_k},
\end{aligned}
\end{equation*}
where the last step uses \eqref{prf:the:x-y}. From the definition of
$\xi_{k+1}$ and the settings of $\Delta_k$ in \eqref{settings-Delta}, there
exists $a_2>0$ such that
\begin{equation}\label{ineq:sub-4-2}
    \|\mathrm{diag}([s_j^{k+1}]^2)\mu^k\|\leq a_2\xi_{k+1}.
\end{equation}
Finally, by \eqref{prf:the:x-y} and $\|h(x^{k+1})\|\leq\|h(y^k)\|$,
\begin{equation*}
\|\omega_{k+1}\|\leq\sqrt{\alpha\|h(y^k)\|+\frac{\ell_gL_\mathcal{A}}{2}\|d^{k}\|^2+\frac{\sigma}{4}S_{k}+\frac{\sigma}{8}S_{k-1}+\ell_g\ell_{\mathcal{A}}\hat{\tau}\sqrt{C_2}\|h(y^k)\|},
\end{equation*}
so the definition of $S_k$ and $\xi_{k+1}$ gives $a_3>0$ with
$\|\omega_{k+1}\|\leq a_3\xi_{k+1}$. Combining these bounds with \eqref{in:partial},
\eqref{ineq:sub-1}, \eqref{ineq:sub-2}, \eqref{ineq:sub-3} and
\eqref{ineq:sub-4-2}, there exists $c_2>0$ such that for all sufficiently
large $k$,
\[\mathrm{dist}(0,\partial P(x^{k+1},z^{k+1},\upsilon^{k+1},s^{k+1},\omega_{k+1}))\leq \|\zeta^{k+1}\|\leq c_2\xi_{k+1}.\]
We complete the proof.
\end{proof}

\begin{proposition}\label{prop:kl-cond-2}
     Let $\{(x^k,y^k,d^k,v^k,\lambda^k,\mu^k):k\in\mathbb{N}\}$ be generated by Algorithm \ref{FSIPL-Line search-comp} under the conditions of Lemma \ref{lem:Desc:Q} with $t_k$ and  $\Delta_k$ satisfying \eqref{settings-Delta}.
     Let $P$ be defined by \eqref{def:P} and $(z^k,\upsilon^k,s^k,\omega_k)\in\mathbb{R}^{p}\times\mathbb{R}^m\times\mathbb{R}^p\times\mathbb{R}$ be defined as \eqref{def:nu-omega-xi}, for all $k\geq 2$.
     Then, the following three statements hold.
\begin{enumerate}
    \item[(i)] There exists  $c_1>0$ such that for all $k\geq 2$, 
    \begin{equation}\label{ineq:Desc:P-2}
        c_1\xi_{k+1}^2\leq P(x^{k},z^{k},\upsilon^k,s^{k},\omega_k)-P(x^{k+1},z^{k+1},\upsilon^{k+1},s^{k+1},\omega_{k+1}).
    \end{equation}
\item[(ii)] $\{S_k:k\in\mathbb{N}\}\in\mathcal{S}$.
\item[(iii)] Let $\Lambda$ be the set  of  cluster points of the sequence $(x^{k},z^{k},\upsilon^k,s^{k},\omega_k)$. Then, the limit $P_{\infty}:=\lim_{k\to\infty} P(x^{k},z^{k},\upsilon^k,s^{k},\omega_k)$ exists and for all $(x^{*},z^{*},\upsilon^*,s^{*},\omega_*) \in\Lambda$, $P(x^{*},z^{*},\upsilon^*,s^{*},\omega_*)=P_{\infty}$.
\end{enumerate}   
 \end{proposition}
\begin{proof}
We first prove Item $(i)$. 
  Combining \eqref{def-Q}, \eqref{def:P}, \eqref{def:C}, \eqref{def:D} and \eqref{def:nu-omega-xi}, we have
\begin{equation}\label{P-exp}
P(x^{k},z^{k},\upsilon^k,s^{k},\omega_k)=Q_k+\frac{\sigma}{4}S_{k-1}+\frac{\sigma}{8}S_{k-2}.
\end{equation}
By \eqref{settings-Delta}, it holds for all $k\geq 2$ that
\begin{equation}\label{prop-kl-cond2-ineq-S}  (\alpha+\|\lambda^k\|+\|\mu^k\|)\Delta_k+\ell_g\Delta_{k-1}\leq \frac{\sigma}{8}S_{k-2}.
\end{equation}
Note that the settings of $t_k$ and  $\Delta_{k}$ in \eqref{settings-Delta} satisfying the condition that $\hat{\eta}_2(t_k,\Delta_k)\geq 1$ in Lemma \ref{lem:Desc:Q}, then by \eqref{prop-kl-cond2-ineq-S}, \eqref{Desc:Q} in Lemma \ref{lem:Desc:Q} and the definition of $S_k$, we deduce that
\begin{equation*}
\begin{aligned}
           Q_k+\frac{\sigma}{8}S_{k-2}  
         \geq Q_{k+1}+\frac{\sigma}{2}(\|d^{k}\|^2+\overline{\tau}^2\|h(y^k)\|+\|v^k\|^2).
         \end{aligned}
\end{equation*}
This together with \eqref{P-exp} and the definition of $S_k$ yields that
\begin{equation}\label{ineq-desc-P}
\begin{aligned}
    &P(x^{k},z^{k},\upsilon^k,s^{k},\omega_k)-P(x^{k+1},z^{k+1},\upsilon^{k+1},s^{k+1},\omega_{k+1})\\
    \geq& \frac{\sigma}{4}\bigl(\|d^k\|^2+\|v^k\|^2+\overline{\tau}^2\|h(y^k)\|\bigr)+\frac{\sigma}{8}S_{k-1}.
    \end{aligned}
\end{equation}
Leveraging this with the definition of $\xi_{k+1}$, we deduce that  there exists $c_1>0$ such that \eqref{ineq:Desc:P-2} holds for all $k\geq 2$.

Next, we prove item $(ii)$. Under the setting \eqref{settings-Delta}, $\eta_k\equiv\overline{\eta}=1$ and $\tau_k\equiv \overline{\tau}$ satisfy \eqref{ub:eta,tau-comp} in Proposition~\ref{Prop:hat_x,y}. Since $x^0\in\mathcal{M}$, Proposition~\ref{Prop:hat_x,y}(i) gives $x^{k+1}\in\mathcal M_{\kappa}\subset\mathcal M_\theta$ and $y^k\in\mathcal M_{\theta}$ for all $k\ge0$. By Lemma~\ref{lemma:d^k-bound-comp}, the sequences $\{d^k\}$,
$\{v^k\}$, $\{\lambda^k\}$, and $\{\mu^k\}$ are bounded. Moreover, Proposition \ref{Prop:hat_x,y} and Assumption \ref{assp:N}(i)--(ii) imply that $\{x^{k+1}-y^k\}$ is bounded. Consequently, \eqref{ineq:Ax-} shows that
\[
\|\mathcal A(x^k)-z^k\|
\leq
\ell_{\mathcal A}\|x^k-y^{k-1}\|
+\frac{L_{\mathcal A}}{2}\|d^{k-1}\|^2+\Delta_{k-1}
\]
is uniformly bounded.

By \eqref{ineq-desc-P}, the sequence
$\{P(x^k,z^k,\upsilon^k,s^k,\omega_k):k\ge2\}$ is nonincreasing. Since
$(x^k,\upsilon^k)\in \mathcal C$ and $(x^k,z^k,s^k)\in \mathcal D$, we have
\[
    P(x^k,z^k,\upsilon^k,s^k,\omega_k)
    =
    f(x^k)+g(z^k)+\omega_k^2 .
\]
Hence, using the Lipschitz continuity of $g$ and $\omega_k^2\geq0$, for all $k\ge2$,
\[
\begin{aligned}
F(x^k)
&=f(x^k)+g(\mathcal A(x^k))\\
&\leq f(x^k)+g(z^k)+\ell_g\|\mathcal A(x^k)-z^k\|\\
&\leq P(x^2,z^2,\upsilon^2,s^2,\omega_2)
+\ell_g\sup_{j\ge2}\|\mathcal A(x^j)-z^j\|.
\end{aligned}
\]
Assumption \ref{assp:N}(vi) now implies that $\{x^k\}$ is bounded. The Lipschitz continuity of $\mathcal A$ on $\mathcal M_{\theta}$ (Assumption~\ref{assp:N}(iv)) and the boundedness of $\{v^k\}$ then imply that
$\{z^k=\mathcal A(x^{k-1})+v^{k-1}\}$ is bounded. The definitions in
\eqref{def:nu-omega-xi}, together with the already established bounds, show that
$\{\upsilon^k\}$, $\{s^k\}$, and $\{\omega_k\}$ are bounded as well. Thus the extended sequence
$\{(x^k,z^k,\upsilon^k,s^k,\omega_k):k\ge2\}$ is bounded. Since $f$ and $g$ are continuous, the sequence of values of $P$ along this bounded sequence is bounded from below. Summing the descent inequality
\eqref{ineq-desc-P} from $k=2$ to $N$ gives
\[
\begin{aligned}
    \frac{\sigma}{8}\sum_{k=2}^{N}S_{k-1}
    \le\;&
    \sum_{k=2}^{N}
    \Bigl[
    P(x^k,z^k,\upsilon^k,s^k,\omega_k)
    -P(x^{k+1},z^{k+1},\upsilon^{k+1},s^{k+1},\omega_{k+1})
    \Bigr]  \\
    =\;&
    P(x^2,z^2,\upsilon^2,s^2,\omega^2)
    -
    P(x^{N+1},z^{N+1},\upsilon^{N+1},s^{N+1},\omega_{N+1}).
\end{aligned}
\]
The right-hand side is bounded above uniformly in $N$ by the lower
boundedness just proved. Letting $N\to\infty$, we obtain $\sum_{k=1}^{\infty}S_k<\infty$. Thus $\{S_k:k\in\mathbb N\}\in\mathcal S$.

We now prove item  $(iii)$. The boundedness established in the proof of item $(ii)$ implies that the cluster point set $\Lambda$ is nonempty and compact. Moreover, by the lower boundedness proved above and the descent inequality \eqref{ineq-desc-P}, the sequence
    $\{P(x^k,z^k,\upsilon^k,s^k,\omega_k):k\ge2\}$
is nonincreasing and bounded from below. Therefore, the limit
\[
    P_\infty
    :=
    \lim_{k\to\infty}P(x^k,z^k,\upsilon^k,s^k,\omega_k)
\]
exists. It remains to show that $P$ is constant on $\Lambda$. Let
$(x^*,z^*,\upsilon^*,s^*,\omega_*)\in\Lambda$. Then there exists a subsequence
$\{k_j\}$ such that
\[
    (x^{k_j},z^{k_j},\upsilon^{k_j},s^{k_j},\omega_{k_j})
    \to
    (x^*,z^*,\upsilon^*,s^*,\omega_*).
\]
Since
\[
    (x^{k_j},\upsilon^{k_j})\in \mathcal C,\qquad
    (x^{k_j},z^{k_j},s^{k_j})\in \mathcal D
\]
for all $j$, and since $\mathcal C$ and $\mathcal D$ are closed, we have
\[
    (x^*,\upsilon^*)\in \mathcal C,\qquad
    (x^*,z^*,s^*)\in \mathcal D.
\]
Thus $(x^*,z^*,\upsilon^*,s^*,\omega_*)\in\operatorname{dom}P$. Along the subsequence $\{k_j\}$, it holds that
\[
    P(x^{k_j},z^{k_j},\upsilon^{k_j},s^{k_j},\omega_{k_j})
    =
    f(x^{k_j})+g(z^{k_j})+\omega_{k_j}^2 .
\]
Passing to the limit and using the continuity of $f$ and $g$, we obtain
\[
\begin{aligned}
    P(x^*,z^*,\upsilon^*,s^*,\omega_*)
    &=
    f(x^*)+g(z^*)+\omega_*^2  \\
    &=
    \lim_{j\to\infty}
    P(x^{k_j},z^{k_j},\upsilon^{k_j},s^{k_j},\omega_{k_j})
    =
    P_\infty .
\end{aligned}
\]
Since $(x^*,z^*,\upsilon^*,s^*,\omega_*)\in\Lambda$ was arbitrary, $P$ is constant on $\Lambda$. This proves item $(iii)$.

Proposition \ref{prop:kl-cond-1} verifies the relative-error condition for the auxiliary function $P$, while Proposition \ref{prop:kl-cond-2} establishes the descent condition and continuity property on the cluster set. Therefore, all assumptions of Proposition \ref{prop:kl-converge} are satisfied, and we obtain the following full-sequence convergence result.
\begin{theorem}\label{conv-KL}
      Let $\{(x^k,y^k):k\in\mathbb{N}\}$ be generated by Algorithm \ref{FSIPL-Line search-comp} under the conditions of Lemma \ref{lem:Desc:Q} with $t_k$ and  $\Delta_k$ satisfying \eqref{settings-Delta}. If $P$ defined in \eqref{def:P} satisfies the KL property, then both sequences $\{x^k:k\in\mathbb{N}\}$ and $\{y^k:k\in\mathbb{N}\}$  converge to a stationary point of problem \eqref{prob:primal}.  
\end{theorem}
\proof{Proof.}
Invoking Proposition \ref{prop:kl-converge}, Proposition \ref{prop:kl-cond-1} and Proposition \ref{prop:kl-cond-2}, we deduce that $\sum_{k=0}^{\infty}\xi_k<\infty$. By the definition of $\xi_{k+1}$ and \eqref{ineq:x-xi}, it holds that
\begin{equation}\label{ineq-seq-conv-1}
   \lim_{k\to\infty} v^k=0, ~~\lim_{k\to\infty} d^k=0,~~\lim_{k\to\infty} h(y^k)=0,~~\sum_{k=0}^{\infty} \|x^{k+1}-x^k\|<\infty.
\end{equation}
This together with \eqref{step2} yields that 
\begin{equation}\label{ineq-seq-conv-2}
    x^*:=\lim_{k\to\infty} x^k=\lim_{k\to\infty} y^k\in\mathcal{M}.
\end{equation}
Since \[\sum_{k=0}^{\infty}\Delta_k\le c_\Delta\sum_{k=0}^{\infty}\rho_k<\infty\] by \textbf{Line~2} in Algorithm \ref{FSIPL-Line search-comp}, it follows from Theorem \ref{the:convger-comp} that $x^*$ is a stationary point of problem~\eqref{prob:primal}.
\end{proof}

We remark that as pointed out in \citet{attouch2010proximal}, all proper semialgebraic functions satisfy the KL property. Consequently, Theorem \ref{conv-KL} is applicable when the involved functions $f$, $g$, $\mathcal{A}$ and $h$ are semialgebraic, which implies that $P$ is semialgebraic. Indeed, the objective functions are semialgebraic in a wide range of sparse optimization problems,
including the problems \eqref{SPCA} and \eqref{SSC}.

\section{Numerical experiments}\label{sec:numerical}
In this section, we demonstrate the efficiency of the proposed FSIPL algorithm by applying it to the SPCA problem \eqref{SPCA} and the SSC problem \eqref{SSC}. All the experiments are implemented in MATLAB R2025a and conducted on a MacBook Air with an Apple M4 chip and 16GB of RAM.

We first specify some implementation details of the proposed algorithm. The parameter choices below are practical updates designed to meet the proportional error budget $\Delta_k\le c_\Delta\rho_k$ with $\{\rho_k\}\in\mathcal S$ required by the analysis in Section~\ref{subsec-sub-convergence}; the stronger full-step tolerance setting used to establish the Kurdyka--\L{}ojasiewicz full-sequence result in Section \ref{sec:converge} is not imposed in the implementation.  We fix $(\gamma,\overline{\eta},\overline{\tau},\sigma)=(0.5,1,1,2)$ and set $\kappa=0.3$, $\theta=0.6$  throughout the tests. 
In view of \eqref{append-Stiefel}, it is not hard to obtain that $\frac{M_1}{C_1}=\frac{\sqrt{1+\theta}}{2(1-\theta)}(\ell_f+\ell_g\ell_{\mathcal{A}})$, where $M_1>0$ is defined in Proposition \ref{prop:Desc:Phi-comp}. Combining these with the aforementioned choices of $\sigma$, $\theta$ and the facts that $\hat{\tau}\leq \overline{\tau}=1$, we deduce that \eqref{def-alpha} holds if $\alpha>\frac{\sqrt{10}}{2}(\ell_f+\ell_g\ell_{\mathcal{A}})$. Consequently, we choose $\alpha=6(\ell_f+\ell_g\ell_{\mathcal{A}})$ for both problems \eqref{SPCA} and \eqref{SSC}.
It is easy to verify that we can take  $\ell_f=(1+\theta)\|B\|^2_F$, $\ell_g=\mu \sqrt{np}$ and $\ell_{\mathcal{A}}=1$ for problem \eqref{SPCA}, while $\ell_f=2(1+\theta)\|L\|_F$, $\ell_g=\mu n$ and $\ell_{\mathcal{A}}=2(1+\theta)$ for problem \eqref{SSC}.Finally, we set $\rho_0=15p\alpha$ and $\Delta_0=0.5$. Since the coupled scheme solves the primal--dual subproblem at every outer iteration, the direction and stepsize produced at the immediately preceding outer iteration are always available and are reused to scale the first heuristic tolerance. We then set, for $k\geq1$,
\[\rho_k=\frac{15p\alpha}{k^{1.01}},\qquad
\Delta_k=\min\left\{c_1\frac{\|d^{k-1}+\nabla\mathcal{A}(x^{k-1})^{\top}v^{k-1}\|}{B_{k-1}},\,\frac{c_2}{k^{c_3}},\,\overline{\Delta}\right\},\]
where $c_1>0,~c_2>0$, $c_3\geq1.01$ and $\overline{\Delta}>0$ vary across the two problems. The constant $c_\Delta>0$ appearing in the analysis is not prescribed or tuned in the implementation. Indeed, $\rho_k=15p\alpha/k^{1.01}$ while $\Delta_k\leq c_2/k^{c_3}$ with $c_3\geq1.01$, so the ratio $\Delta_k/\rho_k$ is bounded for $k\geq1$ (it is at most $c_2/(15p\alpha)$ when $c_3=1.01$ and tends to zero when $c_3>1.01$); together with the fixed initial values $\Delta_0=0.5$ and $\rho_0=15p\alpha$, this shows that there exists a constant $c_\Delta>0$ such that $\Delta_k\leq c_\Delta\rho_k$ for all $k$. Its explicit value is immaterial and is needed only for the convergence analysis.
 Let $0<\underline{B}\leq\overline{B}$ and 
\begin{equation*}\label{t_k}
    B_k:=
    \min\left\{\max\left\{\underline{B},\left|\frac{\|x^{k}-x^{k-1}\|^2}{\langle x^{k}-x^{k-1}, \Delta R^k \rangle}\right|\right\},\overline{B}\right\},
\end{equation*}
where \[\Delta R^k=\mathrm{grad} f(x^k)-\mathrm{grad}f(x^{k-1})+\bigl(\nabla h(x^k)-\nabla h(x^{k-1})\bigr)^{\top}(\lambda^{k-1}-\lambda^{k-2})\] with $\mathrm{grad} f(x^k)=\bigl(I_n-\frac{1}{4}\nabla h(x^k)^{\top}\nabla h(x^k)\bigr)\nabla f(x^k)$ and $\lambda^{-1}=0$. When $\langle x^k-x^{k-1},\Delta R^k\rangle=0$ (or both the numerator and denominator vanish), the quotient is undefined and we set $B_k=B_{k-1}$ before applying the truncation to $[\underline B,\overline B]$. For proximal parameter $t_k$, motivated by the Barzilai-Borwein (BB) stepsize (\citet{iannazzo2018riemannian,2013AWen,gao2018new}), we adopt $B_0=\min\big\{\overline{B},\max\{1/L_f,\underline{B\}\}}$ and for $k\geq 0$, we set $t_k=2B_k$ and $(\underline{t},\overline{t})=(2\underline{B},2\overline{B})$ for problem \eqref{SPCA}, and set  $t_k=B_k$ and $(\underline{t},\overline{t})=(\underline{B},\overline{B})$ for problem \eqref{SSC}. Specifically, we set $(\underline{B,\overline{B}})=(10^{-3},10^{5})$ throughout the tests. The dual \eqref{prob:dual-comp} (and \eqref{dual-A=I} for SPCA) is minimized by an accelerated proximal-gradient method with a closed-form proximal step. Each inner solve is warm-started from the dual iterate returned at the preceding outer iteration. The computed primal-dual direction is kept fixed during the subsequent outer backtracking. The inner method stops as soon as the gradient-norm criterion \eqref{inexact-solution-comp} holds.

\subsection{Experiments on SPCA}\label{SPCA-exp}
In this subsection, we compare the proposed FSIPL method with the ARPG algorithm in \citet{huang2022riemannian} and the SLPG algorithm in \citet{liu2024penalty} for solving the SPCA problem \eqref{SPCA}.  

In the tests, we generate synthetic datasets as in \citet{chenProximalGradientMethod2020} with settings $m=50$ and $n=2000$. For FSIPL, we set $c_1=c_2=p^2$, $c_3=1.01$, $\overline{\Delta}=0.5$,  and terminate the algorithm when
\begin{equation}\label{terminate}
   \max\{\frac{\|d^k\|}{B_k}, \|d^k\|,\|h(x^k)\|\}<\varepsilon.
\end{equation}
with $\varepsilon=\min\{10^{-4},10^{-8}np\}$. For this problem, $\mathcal A=I_n$, the reduced dual construction in \eqref{dual-A=I}--\eqref{def-d^k-A=I} sets $v^k=d^k$. Hence the consistency residual $\nabla\mathcal A(x^k)d^k-v^k$ is identically zero. With $t_k=2B_k$, the first residual component equals $\|d^k+\nabla\mathcal A(x^k)^\top v^k\|/t_k=2\|d^k\|/(2B_k)=\|d^k\|/B_k$ and $\|v^k\|=\|d^k\|$, which gives the computable criterion \eqref{terminate}.

The parameters of ARPG and SLPG  are the same as those in \citet{huang2022riemannian} and \citet{liu2024penalty} respectively. Both of the  two algorithms are terminated when their respective stopping conditions are satisfied with the same  $\varepsilon$. Besides, we start all the algorithms from the same random initial points on $\mathrm{St}(n,p)$, and set their maximum iteration number as 5000. As suggested in \citet{liu2024penalty}, for FSIPL and SLPG, we perform a postprocessing step by projecting the final iterate onto $\mathcal{M}$. This step is not counted in Iter., and its cost is excluded from the reported CPU time.

 We present the average results over 20 runs on different settings of $(p,\mu)$ in Table \ref{tab:SPCA-1}, where the objective values (Obj.), the CPU time in seconds (Time) and the number of iterations (Iter.) are reported for all three algorithms. It can be  observed that all the three algorithms achieve similar objective values, while the proposed FSIPL method takes much less CPU time.
 
\begin{table}[htbp]
\centering
\caption{Results for SPCA}
\label{tab:SPCA-1}
\setlength{\tabcolsep}{2.8pt}
\small
\begin{tabular}{
 c
 S[table-format=-3.3]
 S[table-format=2.3]
 S[table-format=4.0]
 S[table-format=-3.3]
 S[table-format=2.3]
 S[table-format=4.0]
 S[table-format=-3.3]
 S[table-format=2.3]
 S[table-format=4.0]
}
\toprule
\multirow{2}{*}{\makecell{Parameters \\($p$, $\mu$)}} &
\multicolumn{3}{c}{{FSIPL}} &
\multicolumn{3}{c}{{SLPG}} &
\multicolumn{3}{c}{{ARPG}} \\
\cmidrule(lr){2-4} \cmidrule(lr){5-7} \cmidrule(lr){8-10}
& {Obj.} & {Time} & {Iter.} & {Obj.} & {Time} & {Iter.} & {Obj.} & {Time} & {Iter.} \\
\midrule
(20, 0.2) & -340.505 & \textbf{1.279} & 1138 & -340.469 & 2.594 & 1167 & -340.491 & 6.921 & 1002 \\
(20, 0.3) & -279.674 & \textbf{1.087} &  925 & -279.641 & 1.943 &  751 & -279.638 & 7.705 &  845 \\
(20, 0.4) & -222.518 & \textbf{0.943} &  820 & -220.522 & 2.010 &  759 & -222.447 & 6.710 &  901 \\
(20, 0.5) & -168.518 & \textbf{0.959} &  836 & -166.860 & 2.167 &  715 & -168.551 & 8.298 &  820 \\
(20, 0.6) & -119.002 & \textbf{0.835} &  723 & -119.022 & 1.979 &  605 & -119.104 & 9.140 &  766 \\
\midrule
(5,  0.5)  &  -50.081 & \textbf{0.070} &  255 &  -50.079 & 0.162 &  218 &  -50.151 & 0.196 &  344 \\
(15, 0.5)  & -133.865 & \textbf{0.559} &  607 & -133.711 & 1.392 &  515 & -133.813 & 5.077 &  671 \\
(25, 0.5)  & -200.099 & \textbf{1.292} &  845 & -199.893 & 3.246 &  764 & -200.234 & 11.237 &  851 \\
(35, 0.5)  & -250.951 & \textbf{4.608} & 1608 & -250.730 & 11.658 & 1147 & -250.879 & 41.714 & 1292 \\
(45, 0.5)  & -286.949 & \textbf{11.566} & 2131 & -283.427 & 47.221 & 1581 & -286.963 & 97.333 & 1344 \\
\bottomrule
\end{tabular}
\end{table}

\subsection{Experiments on SSC}\label{SSC-exp}
In this subsection, we compare the proposed FSIPL method with the RADA-PGD algorithm in  \citet{xu2026riemannian} and the MPGDA algorithm in \citet{xie2025proximal} for solving the SSC problem \eqref{SSC}. Note that both RADA-PGD and MPGDA algorithms are applied to some equivalent min-max reformulation rather than \eqref{SSC} itself. 

In the tests, we generate synthetic datasets as in \citet{2016ConvexSparseSpectralClustering}, where the data points $\{a_i\}_{i=1}^n$, with $a_i\in\mathbb R^{n}$ independently drawn from $\mathcal N(0,I_n)$, give $W_{ij}=|\langle a_i,a_j\rangle|$. For FSIPL, we set $c_1 = 1/(2p)$, $c_2 = 0.5p$, $c_3 = 1.1$ and $\overline{\Delta}=10\mu$. For the RADA-PGD and MPGDA algorithms, we adopt the parameters as suggested in \citet[Section~6.3]{xie2025proximal}. The initial point of RADA-PGD is chosen as $X_0 X_0^T$, where $X_0$ consists of $p$ eigenvectors associated with $p$ smallest eigenvalues of $L$. We also use the above $X_0$ as the initial points of MPGDA and FSIPL. The FSIPL algorithm is terminated  when
\begin{equation*}\label{terminate-ssc}
   \max\{\frac{\|d^k+\nabla\mathcal{A}(x^k)^{\top}v^k\|}{B_k}, \|v^k\|,\|h(x^k)\|\}<\varepsilon.
\end{equation*}
with \mbox{$\varepsilon = 10^{-4}$}, while the RADA-PGD and MPGDA algorithms are terminated when they find an $\varepsilon$-game-stationary point of their respective minimax formulations with the same $\varepsilon$.  Moreover, for all the compared algorithms, we set the maximum iteration number as 1000. For FSIPL we apply the optional postprocessing projection of the final iterate onto $\mathcal{M}$ described above. This step is not counted in Iter., and its cost is excluded from the reported CPU time.
\begin{table}[!htbp]
\centering
\caption{Results for SSC}
\label{tab:SSC}
\setlength{\tabcolsep}{2.8pt}
\small
\begin{tabular}{
 c
 S[table-format=2.3]
 S[table-format=2.3]
 S[table-format=3.0]
 S[table-format=2.3]
 S[table-format=2.3]
 S[table-format=3.0]
 S[table-format=2.3]
 S[table-format=2.3]
 S[table-format=3.0]
}
\toprule
\multirow{2}{*}{\makecell{Parameters \\($p$, $\mu$)}} &
\multicolumn{3}{c}{{FSIPL}} &
\multicolumn{3}{c}{{MPGDA}} &
\multicolumn{3}{c}{{RADA\_PGD}} \\
\cmidrule(lr){2-4} \cmidrule(lr){5-7} \cmidrule(lr){8-10}
& {Obj.} & {Time} & {Iter.} & {Obj.} & {Time} & {Iter.} & {Obj.} & {Time} & {Iter.} \\
\midrule
(5,  0.5)  & 7.487 & \textbf{0.762} & 31  & 7.487 & 1.430 & 118 & 7.490 & 4.738 & 570 \\
(10, 0.5)  & 14.974 & \textbf{0.804} & 29  & 14.974 & 1.542 & 120 & 14.978 & 5.105 & 411 \\
(15, 0.5)  & 22.460 & \textbf{1.318} & 30  & 22.462 & 1.913 & 117 & 22.466 & 6.919 & 332 \\
(20, 0.5)  & 29.947 & \textbf{1.562} & 29  & 29.949 & 1.850 & 117 & 29.953 & 6.726 & 292 \\
\midrule
(5,  0.2)  & 5.987 & \textbf{0.631} & 25  & 5.987 & 1.038 & 102 & 5.988 & 3.667 & 488 \\
(5,  0.4)  & 6.987 & \textbf{0.672} & 29  & 6.987 & 1.408 & 114 & 6.989 & 4.949 & 548 \\
(5,  0.6)  & 7.987 & \textbf{0.765} & 31  & 7.987 & 1.568 & 122 & 7.991 & 5.167 & 579 \\
(5,  0.8)  & 8.987 & \textbf{0.865} & 33  & 8.987 & 1.599 & 127 & 8.992 & 4.959 & 579 \\
(5,  1.0)  & 9.988 & \textbf{0.956} & 35  & 9.987 & 1.626 & 133 & 9.994 & 4.844 & 597 \\
\bottomrule
\end{tabular}
\end{table}
The computational results averaged over 50 runs on different settings of $(p,\mu)$ with $n=500$ data points are summarized in Table \ref{tab:SSC}, where we report the obtained objective function value of problem \eqref{SSC} (Obj.), the CPU time in seconds (Time), and the number of outer iterations (Iter.).  We observe that all the compared algorithms attain comparable objective function values for different number of groups $p$ and regularization parameters $\mu$.  Meanwhile, the FSIPL algorithm consistently outperforms the other two algorithms in terms of CPU time.

\section*{}


%
%
%

\begin{appendices}
 
 \section{Examples of embedded submanifolds that satisfy the items (i) of Assumption \ref{assp:N}.} \label{append} 
\par ~~

\textbf{1. Stiefel Manifold.} In this case,
\[\mathcal{M}=\mathrm{St}(n,p):=\{X\in\mathbb{R}^{n\times p}:X^{\top}X=I_p\},\]
and $h:\mathbb{R}^{n\times p}\to\mathbb{S}^p$ is defined by \[h(X)=X^{\top}X-I_p.\]
 For $X\in\mathbb{R}^{n\times p}$, let $X = U \Sigma V^T$ be the singular value decomposition of $X \in \mathbb{R}^{n \times p}$, where $U \in \mathrm{St}(n,p)$, $V \in \mathrm{St}(p,p)$ and $\Sigma$ is a diagonal matrix with diagonal entries $\sigma_1\ge\sigma_2\ge\cdots\ge \sigma_p \ge 0$, $i=1,2,\dots,p$. Let $0<\theta<1$ and $X\in\mathcal{M}_{\theta}$. The matrix
$h(X)=X^{\top}X-I$ has eigenvalues $\sigma_i^2-1$, and
\[
\|h(X)\|_F=\left(\sum_{i=1}^{p}(\sigma_i^2-1)^2\right)^{1/2}\leq\theta.
\]
Every eigenvalue of a matrix is bounded in absolute value by its Frobenius
norm, so
\begin{equation}\label{ineq-theta}
1-\theta\leq\sigma_i^2\leq1+\theta,\qquad i=1,\dots,p.
\end{equation}
Next, we estimate the Lipschitz constant of $\nabla h$. For any
$W\in\mathbb{R}^{n\times p}$,
\begin{equation*}
\nabla h(X)[W]=X^{\top}W+W^{\top}X,
\end{equation*}
which yields $L_h=2$. Finally, we estimate the eigenvalues of
$H(X):=\nabla h(X)\nabla h(X)^{\top}$. For any $Y\in\mathbb{S}^p$,
\[
H(X)[Y]=2X^{\top}XY+2YX^{\top}X,
\]
and thus
\[
\langle Y,H(X)[Y]\rangle
=4\operatorname{Tr}(Y^{\top}X^{\top}XY).
\]
Since $X^{\top}X$ has eigenvalues $\sigma_i^2$, the last display together
with \eqref{ineq-theta} gives
\[
4(1-\theta)\|Y\|_F^2
\leq\langle Y,H(X)[Y]\rangle
\leq4(1+\theta)\|Y\|_F^2.
\]
Hence all eigenvalues of $H(X)$ lie in $[4(1-\theta),\,4(1+\theta)]$, and
item~(i) of Assumption~\ref{assp:N} holds with
\begin{equation}\label{append-Stiefel}
0<\theta<1,\quad C_1=4(1-\theta),\quad C_2=4(1+\theta).
\end{equation}

\textbf{2. Oblique Manifold.}
In this case,
\[
\mathcal{M}=\mathrm{OB}(n,p):=\{X=[x_1,\dots,x_p]\in\mathbb{R}^{n\times p}:\|x_i\|^2=1,\ i=1,\dots,p\},
\]
where \(x_i\in\mathbb{R}^n\) denotes the \(i\)-th column of \(X\), and \(h:\mathbb{R}^{n\times p}\to\mathbb{R}^p\) is defined by
\[
h(X)=\big(\|x_1\|^2-1,\dots,\|x_p\|^2-1\big)^\top.
\]
Let $0<\theta<1$ and $X\in\mathcal{M}_{\theta}$. Since
\[
\|h(X)\|=\left(\sum_{i=1}^{p}(\|x_i\|^2-1)^2\right)^{1/2}\leq\theta,
\]
and each summand is bounded in absolute value by the norm,
\begin{equation}\label{ineq-theta-OB}
1-\theta\leq\|x_i\|^2\leq1+\theta,\qquad i=1,\dots,p.
\end{equation}
Next, we estimate the Lipschitz constant of $\nabla h$ and the eigenvalues of
$H(X):=\nabla h(X)\nabla h(X)^{\top}$. For any
$y=(y_1,\dots,y_p)^{\top}\in\mathbb{R}^p$,
\[
\nabla h(X)^{\top}y=(2y_1x_1,\dots,2y_px_p),
\]
so $L_h=2$ and
\[
\langle y,H(X)y\rangle
=\|\nabla h(X)^{\top}y\|^2
=4\sum_{i=1}^{p}\|x_i\|^2y_i^2.
\]
Combining this with \eqref{ineq-theta-OB} gives
\[
4(1-\theta)\|y\|^2
\leq\langle y,H(X)y\rangle
\leq4(1+\theta)\|y\|^2.
\]
Thus all eigenvalues of $H(X)$ lie in $[4(1-\theta),\,4(1+\theta)]$, and
item~(i) of Assumption~\ref{assp:N} holds with
\[
0<\theta<1,\quad C_1=4(1-\theta),\quad C_2=4(1+\theta).
\]

\textbf{3. Affine Subspace.} In this case,
\[
\mathcal{M}=\{x\in\mathbb{R}^n:Ax=b\},
\qquad
h(x)=Ax-b,
\]
where $A\in\mathbb{R}^{m\times n}$ has full row rank and $b\in\mathbb{R}^m$. If $m<n$, then $\mathcal{M}$ is an affine subspace of dimension $n-m$ containing a translate of the nontrivial subspace $\ker A$, so that $\mathcal{M}$ is noncompact. Moreover,
\[
H(x)=\nabla h(x)\nabla h(x)^{\top}=AA^{\top}
\]
has the eigenvalues $\sigma_1(A)^2,\dots,\sigma_m(A)^2$, so that the eigenvalue bound in item~(i) of Assumption~\ref{assp:N} holds on the whole of $\mathbb{R}^n$ with
\[
C_1=\sigma_{\min}(A)^2,\qquad C_2=\sigma_{\max}(A)^2.
\]

\textbf{4. Elliptic Cylinder.} In this case,
\[
\mathcal{M}=\{x\in\mathbb{R}^n:\|Bx\|=1\},
\qquad
h(x)=\|Bx\|^2-1,
\]
where $B\in\mathbb{R}^{q\times n}$ has full row rank with $q<n$. Since $\nabla h(x)=2(B^{\top}Bx)^{\top}$, $H(x)=\nabla h(x)\nabla h(x)^{\top}=4\|B^{\top}Bx\|^2$. If $x\in\mathcal{M}_{\theta}$ with $\theta<1$, then $\|Bx\|\in [\sqrt{1-\theta},\sqrt{1+\theta}]$, and since $Bx\in\mathrm{range}(B)$, it holds that
\[  
\|B^{\top}Bx\|^2\in\big[\sigma_{\min}(B)^2\|Bx\|^2,\ \sigma_{\max}(B)^2\|Bx\|^2\big]\subseteq\big[(1-\theta)\sigma_{\min}(B)^2,\ (1+\theta)\sigma_{\max}(B)^2\big].
\]
Therefore item~(i) of Assumption~\ref{assp:N} holds with
\[
0<\theta<1,~~C_1=4(1-\theta)\sigma_{\min}(B)^2,~~C_2=4(1+\theta)\sigma_{\max}(B)^2;
\]

\section{%
  \texorpdfstring
    {The iterates $\lambda^k,~d^k,~v^k,~\mu^k$ generated at the end of Section \ref{subsec-FSIPL procedure} are consistent with \eqref{def:d^k-comp}, \eqref{def:v^k-comp} and satisfy \eqref{inexact-solution-comp}.}
    {The iterates (lambda, d, v, mu) are consistent with equations from Section ...}%
}\label{append-2}
\par ~~

First, it is easy to check that $d^k$ defined in \eqref{def-d^k-A=I} is an optimal solution to
\begin{equation}
\min_{d \in \mathbb{R}^n} \langle \nabla f(x^k), d \rangle + g(x^k + d) + \langle \lambda^k, \nabla h(x^k) d \rangle + \frac{1}{t_k} \|d\|^2,
\end{equation}
which indicates that with $v^k = d^k$, the pair $(d^k, v^k)$ is also an optimal solution to the following constrained convex problem:
\begin{equation}
\begin{aligned}
\min_{d, v \in \mathbb{R}^n} \quad & \langle \nabla f(x^k), d \rangle + g(x^k + v) + \langle \lambda^k, \nabla h(x^k) d \rangle + \frac{1}{2t_k} (\|d\|^2 + \|v\|^2)~
\text{s.t.}~d = v.
\end{aligned}
\end{equation}
By the strong duality of this problem, there exists $\widetilde{\mu}^k \in \mathbb{R}^n$ such that 
\begin{equation*}
\begin{aligned}
    (d^k,v^k)=\argmin_{d\in\mathbb{R}^n,~v\in\mathbb{R}^n} &\langle \nabla f(x^k), d\rangle+g(x^k+v)+\langle \lambda^k,\nabla h(x^k)d\rangle\\
    &+\langle \widetilde{\mu}^k,d-v\rangle+\frac{1}{2t_k}\|d\|^2+\frac{1}{2t_k}\|v\|^2,
    \end{aligned}
\end{equation*}
and thus we deduce that
 \begingroup
\setlength{\abovedisplayskip}{3pt plus 1pt minus 1pt}
\setlength{\belowdisplayskip}{3pt plus 1pt minus 1pt}
\setlength{\abovedisplayshortskip}{0pt plus 1pt}
\setlength{\belowdisplayshortskip}{3pt plus 1pt minus 1pt}
\begin{align}
d^k&=-t_k(\nabla f(x^k)+\nabla h(x^k)^{\top}\lambda^k+\widetilde{\mu}^k),\label{pre-d}\\
 v^k&=\mathrm{prox}_{t_kg}\left(x^k+t_k\widetilde{\mu}^k\right)-x^k, \label{pre-v}
\end{align}
\endgroup
Now compare \eqref{pre-d} with \eqref{def-mu-A=I}, we know that $\mu^k=\widetilde{\mu}^k$. This together with \eqref{pre-d} tells that $d^k$ defined in \eqref{def-d^k-A=I} is consistent with \eqref{def:d^k-comp}, while \eqref{pre-v} and $\mathcal{A}=I_n$ indicate that $v^k$ satisfies \eqref{def:v^k-comp}. Moreover, by a direct calculation and \eqref{GG-comp-0}, we obtain that
\begin{equation*}
\nabla_{\lambda} G_k(\lambda^k,\mu^k)=-\nabla h(x^k)d^k=\nabla\widetilde{G}_k(\lambda^k),\quad \nabla_{\mu} G_k(\lambda^k,\mu^k)=v^k-d^k=0.
\end{equation*}
Combining this and the fact that $\|\nabla\widetilde{G}_k(\lambda^k)\| \leq \Delta_k$, we immediately obtain \eqref{inexact-solution-comp}.

\section{Proof for Proposition \ref{prop:kl-converge}.}
\begin{proof}\label{append-proof-KL}
    Since $\{(u^k,\upsilon^k)\}$ is bounded, the set $ \Lambda$ is nonempty and compact. If there exists $\hat{k}$ such that for all $k\geq\hat{k}+1$, $\psi(u^{k},\upsilon^k)\equiv \psi (u^{\hat{k}},\upsilon^{\hat{k}})$, then $\xi_{k}\equiv0$ and $u^k\equiv u^{\hat{k}}$ for all $k\geq\hat{k}+1$ due to condition (a). The conclusion follows immediately.
    Otherwise, by Lemma \ref{lemma:uniform^kL}, we know that there exists a continuous concave function $ \varphi $ satisfying conditions (i)-(ii) of Definition \ref{lemma:KL} and an integer $k_0\geq\max\{k_1,k_2+1\}$ such that for all $ (u^*,\upsilon^*) \in \Lambda $ and every $ k \geq k_0 $,  
\begin{equation*}
\varphi'\big( \psi(u^k, \upsilon^k) - \psi(u^*, \upsilon^*) \big) \text{dist}\big( 0, \partial \psi(u^k, \upsilon^k) \big) \geq 1. 
\end{equation*}  
 Then, by condition (b) (Relative error), we have
\begin{equation}
\varphi'\big( \psi(u^k, \upsilon^k) - \psi({u}^*, \upsilon^*) \big) \geq \frac{1}{\text{dist}\big( 0, \partial \psi(u^k, \upsilon^k) \big)}\geq \frac{1}{c_2\xi_{k}} .
\end{equation}  
Define 
\[\Omega_{k,k+1}:=\varphi\big( \psi(u^k, \upsilon^k) - \psi({u}^*, \upsilon^*) \big)-\varphi\big( \psi(u^{k+1}, \upsilon^{k+1}) - \psi({u}^*, \upsilon^*) \big).\]
Due to the concavity and monotonicity of $\varphi$, there holds
\begin{equation*}   
\begin{aligned}
    \Omega_{k,k+1}\geq & \varphi'\big( \psi(u^k, \upsilon^k) - \psi(u^*,\upsilon^*) \big)\big( \psi(u^k, \upsilon^k) - \psi(u^{k+1}, \upsilon^{k+1}) \big)\\
    \geq & \frac{ \psi(u^k, \upsilon^k) - \psi(u^{k+1}, \upsilon^{k+1})}{c_2\xi_k}
    \geq \frac{ c_1\xi_{k+1}^2}{c_2\xi_k},
    \end{aligned}
\end{equation*}
which implies
\begin{equation}\label{KL_Theorem_ineq}
    \frac{c_2}{c_1}\Omega_{k,k+1}+\xi_k\geq \frac{\xi_{k+1}^2+\xi_k^2}{\xi_k}\geq 2\xi_{k+1}.
\end{equation}
Summing \eqref{KL_Theorem_ineq} for $k=k_{0},\dots,N-1$, we obtain
\begin{equation*}
    \sum_{k=k_0+1}^N \xi_{k}\leq \xi_{k_0}+\frac{c_2}{c_1}\sum_{k=k_0}^{N-1}\Omega_{k,k+1}\leq \xi_{k_0}+\frac{c_2}{c_1}\varphi\big( \psi(u^{k_{0}}, \upsilon^{k_0}) - \psi({u}^*, \upsilon^*) \big),
\end{equation*}
where the last inequality comes from the non-negativity of $\varphi$.
Therefore, $\sum_{k=1}^{\infty}\xi_k<\infty$ and then  $\xi_{k}\to 0$. This together with conditions (b) and (c), and the closedness of $\partial \psi$ yields that $0\in \partial\psi(u^*,\upsilon^*)$ for any cluster point of $\{(u^k,\upsilon^k):k\in\mathbb{N}\}$. Moreover, since $\|u^{k+1}-u^k\|\leq c_0\xi_{k+1}$, we conclude that
\[\sum_{k=1}^{\infty}\|u^{k+1}-u^{k}\|<\infty,\]
 which completes the proof.
\end{proof}
 \end{appendices}





\bibliography{sn-bibliography}

\end{document}